\documentclass{amsart}

\usepackage{etex}
\usepackage{amsmath, amssymb, amsbsy}
\usepackage{array}
\usepackage{graphpap, color, paralist, pstricks}
\usepackage[mathscr]{eucal}
\usepackage[pdftex]{graphicx}
\usepackage[pdftex,colorlinks,backref=page,citecolor=blue]{hyperref}
\usepackage{pifont}

\usepackage{mathtools}
\usepackage{cleveref}
\usepackage{float}
\usepackage{crossreftools}

\newcounter{bullet}

\usepackage{setspace}
\setdisplayskipstretch{1}

\usepackage{geometry}
\usepackage{comment}
\newtheorem{thm}{Theorem}[section]
\newtheorem{prop}[thm]{Proposition}
\newtheorem{cor}[thm]{Corollary}
\newtheorem{lem}[thm]{Lemma}

\theoremstyle{definition}
\newtheorem{mydef}[thm]{Definition}

\newtheorem{example}[thm]{Example}

\newtheorem{remark}[thm]{Remark}
\newtheorem{obs}[thm]{Observation}

\crefname{lem}{lemma}{lemmas}

\newcommand{\gO}{\Omega}

\newcommand{\bO}{O}

\newcommand{\we}{\omega}
\newcommand{\thre}{T}

\newcommand{\NN}{\mathbb{N}}

\newcommand{\cA}{\mathcal{A} }
\newcommand{\cB}{\mathcal{B} }
\newcommand{\cC}{\mathcal{C} }
\newcommand{\cD}{\mathcal{D} }

\newcommand{\cF}{\mathcal{F} }

\newcommand{\cH}{\mathcal{H} }
\newcommand{\cI}{\mathcal{I} }

\newcommand{\cL}{\mathcal{L} }
\newcommand{\cP}{\mathcal{P} }
\newcommand{\cQ}{\mathcal{Q} }

\newcommand{\cU}{\mathcal{U} }
\newcommand{\cV}{\mathcal{V} }
\newcommand{\cW}{\mathcal{W} }
\newcommand{\bE}{\mathbf{E}}
\newcommand{\bbf}{\mathbf{f}}

\newcommand{\beq}[1]{\begin{equation}\label{#1}}
\newcommand{\enq}[0]{\end{equation}}

\newcommand{\eps}{\epsilon}

\newcommand{\gd}[0]{\delta }

\newcommand{\nin}[0]{\noindent}

\newcommand{\sub}[0]{\subseteq}

\newcommand{\ra}[0]{\rightarrow}

\newcommand{\bin}[0]{\mbox{\rm{Bin}}}

\newcommand{\pr}[0]{\mathbb{P}}

\newcommand{\rot}[0]{\text{rot}}

\newcommand{\qq}[0]{(q/2)^2}

\newcommand{\hq}{\frac{q}{2}}

\renewcommand{\eta}{\left(\left(\frac{q}{2}\right)^2\right)}

\begin{document}

\title{The number of colorings of the middle layers of the Hamming cube}

\author{Lina Li $^*$}
\thanks{$^*$  Department of Combinatorics and Optimization, Univerisity of Waterloo}
\email{lina.li@uwaterloo.ca}

\author{Gweneth McKinley $^\dagger$}
\thanks{$^\dagger$ Department of Mathematics, UC San Diego}
\email{gmckinley@ucsd.edu}

\author{Jinyoung Park $^\ddagger$}
\thanks{$^\ddagger$ Department of Mathematics, Courant Institute of Mathematical Sciences, New York University. Research supported by NSF grant DMS-2153844.}
\email{jinyoungpark@nyu.edu}

\begin{abstract}
For an odd integer $n = 2d-1$, let $\mathcal B_d$ be the subgraph of the hypercube $Q_n$ induced by the two largest layers. In this paper, we describe the typical structure of proper $q$-colorings of $V(\mathcal B_d)$ and give asymptotics on the number of such colorings. 
The proofs use various tools including information theory (entropy), Sapozhenko's graph container method and a recently developed method of M. Jenssen and W. Perkins that combines Sapozhenko's graph container lemma with the cluster expansion for polymer models from
statistical physics.
\end{abstract}

\maketitle
%%%%%%%%%%%%%%%%%%%%%%%%%%%%%%%%%%%

\section{Introduction}\label{sec.intro}
We begin by establishing some foundational definitions and notation, which we will use to state our first main theorem (\Cref{realMT} below). Write $Q_n$ for the $n$-dimensional Hamming cube (the graph with the vertex set $\{0, 1\}^ n$ where two vertices are adjacent if they differ in exactly one coordinate).
Denote by $\cL_k$ the "$k$-th layer" of $Q_n$, i.e., the collection of the vertices in $Q_n$ with $k$ 1's in its coordinates. A coloring in this paper always refers to a proper coloring of vertices of the given graph. 
The main results of this paper concern the typical structure of colorings of the middle layers of $Q_n$, which is a problem suggested by Balogh, Garcia, and the first author \cite{BGL}. 

Following \cite{BGL}, throughout the paper, we always assume that $n$ is odd and let $d=(n+1)/2$. Write $\cB_d$ for the subgraph of $Q_n$ induced on $\cL_{d-1} \cup \cL_d$ (the two middle layers of $Q_n$). Note that $\cB_d$ is $d$-regular, bipartite with the unique bipartition $\cL_{d-1} \cup \cL_d$, and balanced (i.e., $|\cL_{d-1}|=|\cL_d|$). We use $N$ for $|V(\cB_d)| (=2{n \choose d})$, and assume $d \ra \infty$.

Given $q \in \mathbb N$ and a graph $G$, write $C_q(G)$ for the collection of $q$-colorings of $G$ and let $c_q(G)=|C_q(G)|$. Our first main theorem concerns asymptotics for the number of $q$-colorings of $\cB_d$ for all even $q$. 

\begin{thm}\label{realMT}
Let $q \ge 4$ be even. Then we have 
\[
c_q(\cB_d)  = (q/2)^{N}\binom{q}{q/2}\exp\left((1+o(1))f(q, d)\right)
\]
as $d \ra \infty$, where
\[
f(q, d)= N(1 - 2/q)^d + N(1 - 2/q)^{2d}\cdot \frac{1}{2}\left(d(1 - 2/q)^{-2} - d-1\right).
\]
In particular, for $q=4$, 
\[
c_4(\cB_d)  \sim 6\cdot 2^N\exp\left(f(4, d)\right)=6\cdot 2^N\exp\left(N2^{-d} + N2^{-2d}(3d/2 -1/2)\right).
\]
\end{thm}
Theorem ~\ref{realMT} itself does not give an asymptotic formula for general $c_q(\cB_d)$ when $q\ge 6$.
In fact, Theorem ~\ref{realMT} is a direct consequence of a stronger result (Theorem~\ref{MT1'}), which provides a detailed approximation of $c_q(\cB_d)$ for all even $q$.
Roughly speaking, with Theorem~\ref{MT1'}, for any given even $q$, the precise asymptotic formula of $c_q(\cB_d)$ can be computed in a finite time (meaning that the running time does not depend on the input size $d$).
The statement of Theorem~\ref{MT1'} requires some technical definitions arising from {statistical physics}, and is therefore postponed to Section~\ref{sec.realMT}.

Our second main theorem, \Cref{realMT2}, characterizes the typical structure of $q$-colorings of $\cB_d$. Before we state the theorem, we first fill in some background and introduce relevant notation.
\vspace{5pt}

\noindent \textit{History and background.} 
We begin with the number of colorings of the Hamming cube $c_q(Q_n)$. The first result about asymptotics of $c_q(Q_n)$ was given in 2003 by Galvin~\cite{G} for the case $q=3$. 
A structural characterization of $C_q(Q_n)$ was later obtained by Engbers and Galvin~\cite{EG} in 2012 (in fact, their result is written in a more general context of graph homomorphisms of discrete tori); this also motivated them to propose a conjecture on the asymptotics of $c_q(Q_n)$ for larger $q$. The next case, $q=4$, was established by Kahn and the third author~\cite{KPq}, affirming the conjecture of Engbers and Galvin.
Recently, Jenssen and Keevash~\cite{JK} settled the conjecture of Engbers and Galvin for any $q$, also providing  a framework which can be used to compute the asymptotics of $c_q(Q_n)$ in finite time for any $q$.

Motivated by the problem of enumerating maximal antichains in Boolean lattices,
Duffus, Frankl, and R\"{o}dl~\cite{duffus2011maximal} initiated the study of related enumeration problems on consecutive layers of $Q_n$. Answering their questions,
Ilinca and Kahn~\cite{ilinca2013counting} determined logarithmic asymptotics of the number of maximal independent sets in the two consecutive layers of $Q_n$ (while the asymptotics itself still remains open). Moving to 2021, inspired by the work of Jenssen and Perkins~\cite{JP} that gives much finer asymptotics for the number of independent sets of $Q_n$ (which was originally obtained by Sapozhenko \cite{Sap87}), Balogh, Garcia and the first author~\cite{BGL} determined asymptotics for the number of independent sets in $\cB_d$, the middle two layers of $Q_n$. In that paper, they gave a conjecture for the asymptotics of the number of maximal independent sets in $\cB_d$, and also proposed the problem of counting $q$-colorings of $\cB_d$.
We remark that independent sets in consecutive layers of $Q_n$ are also closely related to another classical topic in combinatorics, \textit{intersecting set families}, see e.g., \cite{balogh2015intersecting, balogh2021intersecting, balogh2019structure, frankl2018counting}.

In addition to its own interests, another important motivation to study enumeration problems on $\cB_d$ is that it is a good \textit{entry point} to generalize this line of work -- including counting/analyzing typical structures of graph homomorphisms (independent sets, proper colorings, Lipschitz functions, etc.), maximal independent sets, linear extensions of a partial order, and so on: see e.g. \cite{JKP, GGJ, PSY} for some related problems -- from $Q_n$ to other bipartite graphs with good expansion properties. Many prior works, limited by the methods used, have relied heavily on the structure of $Q_n$ (see e.g.,~\cite{EG, G, JK, misqn}).
As a subgraph of $Q_n$, the middle two layers --- while sharing many properties with $Q_n$ in common, e.g., regularity, bipartition, good expansion of small subsets --- also exhibit their own structural behavior, which, in some sense, is {"less structured"} compared to $Q_n$. Because of this difference, the approaches that have successfully worked for enumeration problems on $Q_n$ often do not immediately extend to $\cB_d$. 
Therefore, the study of $\cB_d$ may reveal clues to understanding 
more general classes of bipartite graphs, and may lead to progress on related problems. We will return to this point in Section~\ref{sec:overview}.
\vspace{5pt}

\noindent \textit{Notation.} We use $V$ and $E$ for $V(\cB_d)$ and $E(\cB_d)$, respectively. Given $q$, we use $\cQ=\{1,2,\ldots,q\}$ for the collection of $q$ colors and $f$ for a coloring in $C_q(\cB_d)$. Say an ordered partition $(A,B)$ of $\cQ$ is \textit{principal} if $\{|A|,|B|\}=\{\lfloor q/2 \rfloor, \lceil q/2 \rceil\}$, and $A \sub \cQ$ is \textit{principal} if $A$ is a part of a principal partition (i.e., $|A| \in \{\lfloor q/2 \rfloor, \lceil q/2 \rceil\}$).
For a given principal $(A,B)$, we say $f$ \textit{agrees with $(A,B)$ at $v$} if
\[
f_v \in A \Leftrightarrow v \in \cL_{d}
\quad\mbox{and}\quad
f_v \in B \Leftrightarrow v \in \cL_{d-1}
\]
(where $f_v$ is the value of $f$ at $v$). Given a principal $(A,B)$ and $f$, define
\begin{equation}\label{def:flaw}
X_{A,B}(f)=\{v\in V \mid \mbox{$f$ \textit{disagrees} with $(A,B)$ at $v$}\}.
\end{equation}
We say $f$ is a \textit{ground state coloring} if $X_{A,B}(f)=\emptyset$ for some principal $(A, B)$.

For a simple graph $G$ and a set $W \sub V(G)$, we write $G[W]$ for the subgraph of $G$ induced on $W$ (as usual). Say $W \sub V(G)$ is \textit{2-linked} if $G^2[W]$ is connected, where $G^2$ is the square of $G$ (that is, $G^2$ is the graph on $V(G)$ having an edge between each pair of vertices at distance at most 2 in $G$). A \textit{2-linked component} of $G$ is a maximal 2-linked subset of $V(G)$.

For any integer $q \ge 4$, define the \textit{threshold polymer size}, denoted by $T(q)$, to be 
\beq{def:threpolysize} \mbox{the smallest integer $t$ that satisfies $2+t\log(1-2/q)<0$.}\enq
Our second main theorem essentially reveals that almost every coloring is {"close"} to some {ground state coloring}.

\begin{thm}\label{realMT2}
Let $q \ge 4$ be even. Almost all proper $q$-colorings $f$ of $\cB_d$ satisfy the following property: there exists a principal partition $(A,B)$ such that every 2-linked component of $X_{A,B}(f)$ is of size less than $T(q)$. 
\end{thm}
We believe that $T(q)$ is the smallest integer for which Theorem~\ref{realMT2} holds.
In fact, one can obtain much more detailed structural results on $C_q(\cB_d)$ as in \Cref{notMT} below, whose form is inspired by \cite[Theorem 1.2]{EG} and \cite[Theorem 13.2]{JK}. The proof of \Cref{notMT} is almost identical to the proof of \cite[Theorem 13.2]{JK}  (modulo the results in this paper -- specifically, \Cref{lem.step1,lem:capture1,lem:defect}) thus we omit it.

\begin{thm}\label{notMT}
For any even $q\ge 4$ there exists a constant $\xi \in (1 - 2/q,1)$ that satisfies the following property: if $\xi^d\leq s\leq 1/q^2$, then there exists a partition
\[
C_q(\cB_d) = \cD_0\cup \bigcup_{(A, B) \text{ is principal }} \cD_q(A, B)
\]
satisfying the following properties:
\begin{itemize}
\item[(i)] $|\cD_0|\leq \exp(-\Omega(d))c_q(\cB_d)$;

\item[(ii)] For a principal partition $(A, B)$, a coloring $f \in \cD_q(A, B)$, and a color $a \in A$ (resp. a color $b\in B$), the proportion of vertices of $\cL_d$ (resp. $\cL_{d-1}$) colored $a$  (resp. $b$) is within $s$ of $2/q$.

\item[(iii)] For any distinct principal partitions $(A, B)$ and $(A',B')$, 
\[|\cD_q(A, B)| =\left(1  \pm \exp\left(-\Omega\left(Nt\right)\right)\right)|\cD_q(A', B')|\]
where $t:=\min\{s^2, s/d^2\}$.

\item[(iv)] Let $\mu_q$ be the uniform probability measure on $C_q(\cB_d)$.
For each principal partition $(A, B)$, $x\in \cL_{d}$, $y\in \cL_{d-1}$,  $a\in A$ and $b\in B$,
\[
\pr_{\mu_q}\left(f_x=a\mid f\in \cD_q(A, B)\right)=\frac1{q/2}(1 \pm \exp(-\Omega(d)))
\]
and
\[
\pr_{\mu_q}\left(f_y=b\mid f\in \cD_q(A, B)\right)=\frac1{q/2}(1 \pm \exp(-\Omega(d))).
\] 
\end{itemize}
\end{thm}

\section{Overview of the proof}\label{sec:overview}

Informally speaking, the proof of the main theorems consists of the following three steps, each of which requires separate ingredients:
\begin{itemize}
\item[Step 1.] almost all colorings $f$ admit a \textit{ground state};
\item[Step 2.] $f$'s with large \textit{flaws} (i.e., the vertices that disagree with the ground state of $f$) are negligible;
\item[Step 3.] approximately count the total of colorings using Steps 1 \& 2. 
\end{itemize}

Step 1 is inspired by and closely follows the work of \cite{EG, Kahn}, which uses \textit{entropy ideas} initiated by Kahn~\cite{Kahn}. 
For Step 3, we use the \textit{polymer model} and \textit{cluster expansion} method originated from statistical physics. Recently in \cite{JP}, this method was shown to be very powerful in enumeration problems related to $Q_n$. Though it still requires some technical modifications, this cluster expansion-based approach provides a rather \textit{routine framework} for  enumeration problems when it applies; we refer interested readers to~\cite{BGL, JK}.

Perhaps the most interesting part of our proof is Step 2. 
Inspired by the prior works on colorings of $Q_n$ in \cite{JK, KPq}, the proof of this step integrates an entropy approach and \textit{Sapozhenko’s graph container lemma}~\cite{Sap87}.
However,  neither of the previous arguments directly extend to our problem. In particular, the beautiful proof in~\cite{JK} crucially relies on the existence of a special partition of $V(Q_n)$ (\cite[Lemma 8.1]{JK}), but it is not even clear whether such a special vertex partition exists in $\cB_d$.
The {novelty} of our work is that we implement an entropy approach and the graph container framework (not the lemma itself) in such a way that it does not rely on the precise structure of the host graph, but only on its expanding property. We expect that this new approach will also work for more general classes of bipartite graphs with appropriate expansion properties.

Next, we describe more details for each step. As the first step, we show that if $q$ is an even integer, then for almost every $f \in C_q(\cB_d)$ there is some principal $(A,B)$ with which $f$ agrees on all but an exponentially small fraction of the vertices.

\begin{lem}[Step 1: almost all $f$'s admit a ground state]\label{lem.step1} Let $q \ge 4$ be even. 
There exist $\alpha, \beta >0$ such that for all but a $2^{-\alpha d}$-fraction of $f$'s in $C_q(\cB_d)$, there is a principal $(A,B)$ for which
\beq{eq.step1} |X_{A,B}(f)|<2^{-\beta d}N.\enq
\end{lem}

For $f$ and $(A,B)$ as in Lemma \ref{lem.step1}, call $(A,B)$ the \textit{ground state} of $f$. Note that the rhs of \eqref{eq.step1} can be very large (since $N =2{n \choose d}$), so Lemma \ref{lem.step1} by itself is not enough to give the level of precision as in Theorems \ref{realMT}-\ref{notMT}, and therefore requires a refining process in the next step.

For a principal $(A,B)$ and $S \sub V$, define
\beq{def.chi}\chi_{A,B}(S)=\{f \in C_q(\cB_d): X_{A,B}(f) = S\}.\enq
For $g \in \mathbb N$, let 
\[\cH(g)=\{X \sub V: X \mbox{ is 2-linked}, |N(X)|=g\}\]
where $N(X)=\{y \in V: y \text{ is adjacent to some } x \in X\}$.
\begin{lem}[Step 2: $f$'s with large flaws are negligeable]\label{lem.step2}
Let $\beta'>0$ be such that $d2^{-\beta d}\le 2^{-\beta' d}$ for $\beta$ as in \Cref{lem.step1} (and $d$ large enough). 
For any principal partition $(A, B)$ and $g \in [d^{10}, 2^{-\beta' d}N]$,
\beq{eq.step2} \sum_{X \in \cH(g)} \chi_{A,B}(X) =(q/2)^{N}2^{-\gO(g/\log ^2d)}.\enq
\end{lem}
The lower bound $d^{10}$ on $g$ is just a convenient choice and not at all optimal.
With Step 2, we reduce our problem to taking care of $f$'s with a manageable size of flaws (more specifically, a polynomial size). And the contributions from those small flaws can be easily handled using trivial bounds.

As the final step, we use the polymer model and cluster expansion method developed in \cite{JP} to obtain detailed information about the structure of $C_q(\cB_d)$ as in Theorems \ref{realMT}-\ref{notMT}. The heart of this step is verifying the convergence condition (see~Theorem~\ref{KPconv}) for the cluster expansion of the appropriately defined polymer model, in which \Cref{lem.step2} plays the crucial role.

We point out that the assumption that $q$ is even is only required for Lemma~\ref{lem.step1} (see Remark \ref{stark}), hence the result analogous to \Cref{lem.step1} for odd $q$ would enable us to extend Theorems \ref{realMT}-\ref{notMT} for any integer $q$. It is very natural to ask whether the odd $q$ case could be handled by additional ideas or  different arguments, or if the situation is fundamentally different for even and odd numbers of colors. 
We also note that the particular structure of $\mathcal{B}_d$ is used most heavily in Step 1, where our approach relies on the existence of a large matching with desirable structural properties. In contrast, in Steps 2 and 3, we primarily use the expansion properties of $\mathcal{B}_d$. We believe that it may be relatively straightforward to generalize these results to other bipartite graphs with good expansion. 
\vspace{5pt}

\noindent\textit{Organization.} Section \ref{sec.prelim} collects various preliminary materials. Lemmas \ref{lem.step1} and \ref{lem.step2} are proved in Sections~\ref{sec.step1} and \ref{sec.step2}, respectively. In Section \ref{sec.polymer models} we construct polymer models on $\cB_d$ and prove some important properties. 
The main theorems, \Cref{realMT} and \Cref{realMT2} are proved in \Cref{sec.realMT} and \Cref{sec.realMT2}, respectively.
Before we close this section, we introduce more definitions and notation.\vspace{5pt}

\nin \textit{More definitions.} For $f \in C_q(\cB_d)$, $f_U$ denotes the restriction of $f$ to $U$, and $f(U)$ denotes the color set $\{f_v: v \in U\}$. For $x \in V$, $x_i (\in \{0,1\})$ denotes the $i$-th coordinate in $x$. By the \textit{Hamming weight} of $x$, denoted $|x|$, we mean $|\{x_i:x_i=1\}|$. As usual $N_x=\{v: v \sim x\}$ where $v \sim x$ means $v$ and $x$ are adjacent, and then $N(U)=\cup_{u \in U} N_u$, $\partial(U)=N(U) \setminus U$, and $U^+=U \cup N(U)$.  For $U \sub V$, $d_U(v):=|N(v) \cap U|$. For two disjoint subsets $W_1,W_2$ of $V$, $\nabla(W_1,W_2):=\{(w_1,w_2) \in E:w_1 \in W_1, w_2 \in W_2\}$. 
Again, $\log$ means $\log_2$, and $\ln$ is the natural logarithm.

\section{Preliminaries}\label{sec.prelim}

\subsection{Basics}

Recall that a \textit{composition} of an integer $n$ is an ordered sequence $\langle a_1,\ldots\rangle $ of positive integers summing to $n$. The $a_i$’s are the \textit{part}s of the composition. Below is a well-known fact.

\begin{prop}\label{prop:decom}
The  number  of  compositions  of $n$ is $2^{n-1}$ and  the  number of compositions with  at  most $b$ parts is $\sum_{i<b}\binom{n-1}{i}<2^{b\log(en/b)}$, when $b< n/2$.
\end{prop}

For $X,Y \sub V(G)$, we say that $Y$ \textit{covers} $X$ if $X \sub N(Y)$ and that $Y$ \textit{mutually covers} $X$ if $X \sub N(Y)$ and $Y \sub N(X)$. Observe that if $Y$ covers $X$ then there is $Y' \sub Y$ that mutually covers $X$.
We will use the following lemma, a special case of a fundamental result of Lov\'{a}sz~\cite{lovasz1975ratio} and Stein~\cite{stein1974two}. 

\begin{lem}\label{lem:cover}
Let $G$ be a bipartite graph with bipartition $P\cup Q$, where $|N(u)|\geq a$ for each $u \in P$ and $|N(v)|\leq b$ for each $v\in Q$. Then there exists some $Q'\subseteq Q$ that covers $P$ and satisfies
\[
|Q'|\leq \frac{|Q|}{a}\cdot (1 + \ln b).
\]
\end{lem}

\begin{cor}\label{cor:cover}
For any $X\subseteq V (=V(\cB_d))$, there exists a set $Y\subseteq V$ such that $Y$ mutually covers $X$ and $|Y|\leq \frac{|N(X)|}{d}\cdot (1 + \ln d)$.
\end{cor}
\begin{proof}
Let $X_1=X\cap\cL_{d-1}$ and $X_2=X\cap\cL_{d}$. By applying \Cref{lem:cover} for $X_1\cup N(X_1)$ and  $X_2\cup N(X_2)$
separately, we get mutual covers $Y_1$ and $Y_2$ of $X_1$ and $X_2$. Observe that $Y:=Y_1\cup Y_2$ is a mutual cover of $X$ and 
\[
|Y|=|Y_1|+|Y_2|\leq \frac{|N(X_1)|}{d}\cdot (1 + \ln d) + \frac{|N(X_2)|}{d}\cdot (1 + \ln d)=\frac{|N(X)|}{d}\cdot (1 + \ln d).
\]
\end{proof}

The next lemma is used to bound the numbers of certain types of 2-linked sets in $V(\cB_d)$.
It follows from the fact (see e.g. \cite[p.\ 396, Ex.11]{Knuth}) that the
infinite $\Delta$-branching rooted tree contains precisely
\[\frac{{{\Delta n} \choose n}}{(\Delta-1)n+1} \le (e\Delta)^{n-1}\]
rooted subtrees with $n$ vertices.

\begin{lem}\label{lem:numlinkset}
If $G$ is a graph with maximum degree $\Delta$, then the number of $n$-vertex
subsets of $V(G)$ which contain a fixed vertex and induce a connected subgraph is at most 
\beq{eq:numlinkset}(e\Delta)^{n-1}.\enq
\end{lem}

\subsection{Isoperimetry}

We use isoperimetric inequalities on $\cB_d$, which can be easily derived from direct applications of the Lov\'asz version of the Kruskal-Katona Theorem \cite{Katona, Kruskal}. Recall that $N/2=|\cL_d|=|\cL_{d-1}|$.

\begin{thm}[Lov\'asz \cite{Lovasz}]\label{thm.Lovasz} Let $\cA$ be a family of $m$-element subsets of a fixed set $U$ and $\cB$ be the family of all $(m-q)$-element subsets of the sets in $\cA$. If $|\cA|={x \choose m}$ for some real number $x$, then $|\cB| \ge {x \choose m-q}$.
\end{thm}

\begin{prop}\label{lem:isoper}
Let $X \subseteq \mathcal{L}_d$ or $X\subseteq \mathcal{L}_{d-1}$.
\begin{itemize}
\item[\rm (i)] If $|X|\leq d/4$, then $|N(X)|\geq d|X| - |X|^2/2$.

\item[\rm (ii)] If $|X|\leq d^{10}$, then $|N(X)|\geq d|X|/12$.

\item[\rm (iii)] If $|X|=2^{-\gO(d)}N$, then $|N(X)| \ge (1+\gO(1))|X|$.

\end{itemize}
\end{prop}

\begin{proof} 
(i)$|N(S)|\geq d + (d-1)+\cdots + (d - |S|+1)=d|S| - |S|(|S|-1)/2$.\\
(ii) Let $|S|=\binom{x}{d}$ for some real $x$. Note that in this case, $x\leq d +11$. By \Cref{thm.Lovasz}, we have 
$|N(S)|\geq \binom{x}{d-1}=\binom{x}{d}\frac{d}{x-d+1}=|S|\frac{d}{x-d+1}\geq d|S|/12$.
\\
(iii) If $|X|=2^{-\gO(d)}N$, then $|X|={x \choose d}$ for some $x=(2-\gO(1))d$. By \Cref{thm.Lovasz}, ${x \choose d-1}/{x \choose d}=1+\gO(1)$.
\end{proof}

\subsection{The Graph Container}\label{subsec.graph container}

This section recalls or proves some variants of the graph container lemma due to Sapozhenko~\cite{Sap87}. An excellent exposition on this topic is given in \cite{GS}. Recall that $\cH(g)=\{X \sub V: X \mbox{ 2-linked}, |N(X)|=g\}$. 

\begin{lem}\label{lem:contain1}
For all $g\in\mathbb{N}$, there exists a family $\mathcal{V}=\mathcal{V}(g)\subseteq 2^V$ with 
\[
|\mathcal{V}|\leq 2^{\gO(g \log^2 d/d + d)},
\]
such that for any $X\in\cH(g)$, $\mathcal{V}$ contains a set that mutually covers $X$.
\end{lem}
\begin{proof}
Given a set $X\in\cH(g)$, by \Cref{cor:cover}, there exists a  set $Y\subseteq V$ such that $Y$ mutually covers $X$ and $|Y|\leq \frac{|N(X)|}{d}\cdot (1 + \ln d)=\frac{g}{d}\cdot (1 + \ln d)$. Moreover, observe that $Y$ is 4-linked (i.e. $(\cB_d)^4[Y]$ is connected where $(\cB_d)^4$ is the fourth power of $\cB_d$) as $X$ is 2-linked and $Y$ mutually covers $X$.
We therefore take $\mathcal{V}$ to be the collection of all $4$-linked subsets of $V$ of size at most $\frac{g}{d}\cdot (1 + \ln d)$.
Then by \Cref{lem:numlinkset},
\[
|\mathcal{V}|\leq N\sum_{\ell\leq\frac{g}{d}\cdot (1 + \ln d)}(ed^4)^{\ell}=2^{O(g\log^2 d/d + d)}.
\]
\end{proof}

Let $\psi=\psi(d)>0$. Following \cite{GS, JK}, define a \textit{$\psi$-approximating pair} for $X \sub V$ to be a pair $(F,S) \in 2^V \times 2^V$ satisfying
\beq{approx1} F \sub N(X), S \supseteq X;\enq
\beq{approx2} d_{V \setminus F} (u) \le \psi \quad \forall u \in S;\enq
and
\beq{approx3} d_S(v) \le \psi \quad \forall v \in V \setminus F.\enq

\begin{lem}\label{lem:contain2}
Let $\psi\leq d/2$. For each $Y\subseteq V$ there exists a family $\mathcal{W}=\mathcal{W}(Y, g)\subseteq 2^V\times 2^V$ with 
\[
|\mathcal{W}|\leq 2^{O(g \log d/\psi)}
\]
such that any $X\subseteq V$ mutually covered by $Y$ and $|N(X)|=g$ has a $\psi$-approximating pair in $\mathcal{W}$.
\end{lem}
\nin The proof of Lemma~\ref{lem:contain2} is almost identical to the proof of \cite[Lemma 7.5]{JK}, so we omit.
\begin{lem}\label{lem:container}
Let $\psi \le d/2$. There is a family $\cU=\cU(g) \sub 2^V \times 2^V$ with
\[|\cU|\le 2^{O(g\log^2 d/d + g\log d/\psi+d)}\]
such that every $X \in \cH(g)$ has a $\psi$-approximating pair in $\cU$.
\end{lem}

\begin{proof}
Take $\cU(g)=\bigcup_{Y \in \cV(g)}\cW(Y,g)$ where $\cV(g)$ and $\cW(Y,g)$ are the families in Lemmas~\ref{lem:contain1}- \ref{lem:contain2}.
\end{proof}

For a set $X\subseteq V$, define $X_1:=X \cap \cL_{d-1}$, and $X_2:=X \cap \cL_d$. The next observation says that a $\psi$-approximating pair for $X$ naturally induces $\psi$-approximating pairs for $X_1$ and $X_2$.

\begin{obs}\label{prop:SFappro1}
Let $(F,S)$ be a $\psi$-approximating pair for $X \sub V$. Let $F_1:=F \cap \cL_d$ and $S_1:=S \cap \cL_{d-1}$. 
Then $(F_1, S_1)$ is a $\psi$-approximating pair for $X_1$; i.e., $(F_1,S_1)$ satisfies \eqref{approx1}, \eqref{approx2}, and \eqref{approx3}. The same holds for $X_2$ with $F_2 :=F \cap \cL_{d-1}$ and $S_2:=S \cap \cL_d$.
\end{obs}

\begin{lem}
For $(F_i,S_i)$ ($i=1,2$) as above,
\beq{S.F.gap} |S_i|\le |F_i|+2|N(X_i)|\psi/(d-\psi).\enq
In particular, if $\psi \ll d$ then
\beq{sf.ub} |S_i| \le (1+o(1))|N(X_i)|.\enq
\end{lem}

\begin{proof} 
This proof of \eqref{S.F.gap} is almost identical to the proof of \cite[Lemma 5.3]{GS} (with slightly different parameters).
We use $X, F, S$ for $X_i, F_i, S_i$ for simplicity. First observe that
\[(d|S|-\psi|S \setminus X|=) \; d|X|+(d-\psi)|S \setminus X| \le |\nabla(S,N(X))|\le d|F|+\psi|N(X) \setminus F|.\]
By comparing the first and the last expressions, we have
\beq{SFbd1} |S| \le |F|+\psi|(N(X) \setminus F) \cup (S \setminus X)|/d.\enq
Also observe that
\[|\nabla(N(X))| \ge |\nabla(N(X) \setminus F)|+|\nabla(F, S \setminus X)| \ge (d-\psi)|(N(X) \setminus F) \cup (S \setminus X)|, \]
which gives
\beq{SFbd2} |(N(X) \setminus F) \cup (S \setminus X)| \le |N(X)|d/(d-\psi).\enq
Now \eqref{S.F.gap} follows from the combination of \eqref{SFbd1} and \eqref{SFbd2}, and \eqref{sf.ub} follows from \eqref{S.F.gap} by noticing $|F_i| \stackrel{\eqref{approx1}}{\le} |N(X_i)|$.
\end{proof}

\subsection{Entropy} \label{Entropy}
We briefly recall relevant entropy background; see e.g. \cite{McEliece} for more detailed introduction.

Let $X, Y$ be discrete random variables. The binary entropy function of $X$ is
\[ H(X)=\sum_x p(x) \log {\frac{1}{p(x)}},\]
where $p(x)=\pr(X=x)$ (and $p \log \frac{1}{p} :=0$ for $p=0$).
The \emph{conditional entropy of $X$ given $Y$} is
\beq{condent}
H(X|Y)=\sum_y p(y)\sum_x p(x|y)\log \frac{1}{p(x|y)}
\enq
(where $p(x|y)=\pr(X=x|Y=y)$).

\begin{lem}\label{entropyprop} For $H(\cdot)$ the binary entropy function, the following properties hold.
\beq{entropyprop1} \mbox{$H(X) \le \log |\mbox{Range}(X)|$, with equality iff $X$ is uniform from its range;}\enq
\beq{entropyprop1.5} H(X|Y) \le H(X) ;\enq
\beq{entropyprop2} \mbox{$H(X,Y)=H(X)+H(Y|X)$;}\enq
\beq{entropyprop3} \mbox{$H(X_1\dots X_n|Y) \le \sum H(X_i|Y)$ (note $(X_1\dots X_n)$ is a discrete r.v.);} \enq
\beq{entropyprop4} \mbox{if $Z$ is determined by $Y$, then
$H(X|Y) \le H(X|Z)$}; \enq
\beq{entropyprop5} \mbox{if $Z$ is determined by $X$, then
$H(X,Z|Y) = H(X|Y)$.} \enq
\end{lem}

We also need the following version of \emph{Shearer's Lemma} \cite{Sh}.

\begin{lem}\label{lem:Sh}
Let $X=(X_1, \dots, X_N)$ be a random vector, and $\alpha:2^{[N]}\rightarrow \mathbb R^+$
satisfies
\beq{fractiling}
\sum_{A\ni i}\alpha_A =1 \quad \forall i\in [N].
\enq
Then for any partial order $\prec$ on $[N]$,
\beq{Shearer}
H(X)\leq \sum_{A\subseteq [N]}\alpha_AH\big(X_A|(X_i:i \prec A)\big),
\enq
where $X_A=(X_i:i\in A)$ and $i\prec A$ means $i\prec a\  \forall a\in A$.
\end{lem}

\subsection{Polymer models and the cluster expansion}\label{subsec.polymer intro} 
This section gives a brief introduction to two tools from statistical physics, \textit{polymer models} and \textit{the cluster expansion}, in the language of graph theory. For more general exposition and applications of polymer models, we refer interested readers to~\cite{fernandez2007cluster, JP, kotecky1986cluster}.
\vspace{5pt}

\nin \textit{Abstract polymer model}. 
Let $H_{\cP}$ be a graph  defined on a finite set $\cP$, such that every vertex has a loop edge and there is no multiple edge.
For historical reasons in physics, the vertices $\gamma \in \cP$ are called \textit{polymers}.
Two polymers $\gamma, \gamma'$ are called \textit{adjacent}, denoted by $\gamma\sim \gamma'$, if there is an edge $\gamma\gamma'$ in $H_{\cP}$.
In particular, every polymer is adjacent to itself.
We equip each polymer $\gamma$ with a complex-valued weight $\we(\gamma)$, and refer the weighted graph $(H_{\cP}, \we)$ as \textit{the polymer model}. 
Denote by $\Omega_{\cP}$ the collection of independent sets of $H_{\cP}$ (including the empty set), where loops are allowed.
The \textit{polymer model partition function} 
\begin{equation}\label{def:partitionfunction}
  \Xi(\cP, \we)=\sum_{\Lambda\in \Omega_{\cP}}\prod_{\gamma\in \Lambda}\we(\gamma)
\end{equation}
is essentially the weighted independence polynomial of $(H_{\cP}, \we)$.
Sometimes, by abuse of notation, we also refer to $(\cP,\we)$ or $\Xi(\cP, \we)$ as the polymer model.
\vspace{5pt}

\nin \textit{Cluster expansion}. 
For an ordered multiset $\Gamma=(\gamma_1, \gamma_2, \ldots, \gamma_k)$ of polymers,
we define the \textit{incompatibility graph} $H_{\cP}[\Gamma]$ to be the simple graph defined on $\Gamma$ with the edge set $E=\{\gamma_i\gamma_j: \gamma_i\sim \gamma_j \text{ in } H_{\cP}\}$.
We say an ordered multiset $\Gamma$ of polymers is a \textit{cluster} if it is non-empty, and the incompatibility graph $H_{\cP}[\Gamma]$ is connected.
For example, for two adjacent polymers $\gamma, \gamma'$, the ordered multiset $\Gamma=(\gamma, \gamma', \gamma)$ is a cluster with $H_{\cP}[\Gamma]=K_3$, where $K_3$ is the complete graph on three vertices.

For a simple graph $H$, let $\phi(H)$ be the \textit{Ursell function} of $H$, defined by
\[
\phi(H)=\frac{1}{|V(H)|!}\sum \{(-1)^{e(F)}: \mbox{$F$ is a connected spanning subgraph of $H$}\}.\]

We remark that, from the above definition,
\beq{ursell} \mbox{for any fixed $k \in \NN$, if $|V(H)|\le k$ then $\phi(H)=O_k(1)$.}\enq
The \textit{weight function of a cluster} $\Gamma$ is defined using the Ursell function as
\begin{equation}\label{def:cluweight}
   \we(\Gamma)=\phi(H_{\cP}[\Gamma])\prod_{\gamma\in \Gamma}\we(\gamma).
\end{equation}

Let $\mathcal{C}$ be the set of all clusters on $\cP$. The \textit{cluster expansion} is the formal power series of the logarithm
of the partition function $\Xi(\cP, w)$, which takes the form\footnote{For details of the cluster expansion, we refer  readers to Chapter 5 of \cite{friedli2017statistical}.}
\begin{equation}\label{eq:cluexp}
 \ln\Xi(\cP, \we) = \sum_{\Gamma\in\mathcal{C}}\we(\Gamma). 
\end{equation}

Note that the cluster expansion is an infinite series, despite the finiteness of the polymer model.
A sufficient condition for the convergence of the cluster expansion is given by Koteck\'{y} and Preiss~\cite{kotecky1986cluster} in 1986.

\begin{thm}[Convergence of the cluster expansion]\label{KPconv}
Let $f: \cP \rightarrow [0, \infty)$ and $g: \cP \rightarrow [0, \infty)$ be two functions. Suppose that for all polymers $\gamma\in\cP$,
\begin{equation}\label{eq:kot-thm1}
\sum_{\gamma'\sim \gamma}|\we(\gamma')|\exp\left(f(\gamma')+g(\gamma')\right)\leq f(\gamma),
\end{equation}
then the cluster expansion~(\ref{eq:cluexp}) converges absolutely. 
Moreover, for a cluster $\Gamma\in \cC$, let $g(\Gamma)=\sum_{\gamma\in\Gamma}g(\gamma)$ and write $\Gamma\sim \gamma$ if there exists $\gamma'\in \Gamma$ so that $\gamma\sim \gamma'$. Then for all polymers $\gamma$,
\begin{equation}\label{eq:kot-thm2}
\sum_{\Gamma\in \mathcal{C}, \Gamma\sim \gamma}|\we(\Gamma)|\exp\left(g(\Gamma)\right)\leq f(\gamma).
\end{equation}
\end{thm}

\section{Proof of Lemma \ref{lem.step1}}\label{sec.step1}

This section closely follows the arguments originated in \cite{Kahn} and developed in \cite{EG} (and also adapted in \cite{JK}) with the additional idea of "rotating" $\cB_d$. The arguments in \cite{Kahn,EG} heavily rely on the nice property of the Hamming cube that it can be decomposed into two half cubes with a perfect matching between them (i.e. $Q_n=Q_{n-1} \square K_2$). In their argument, both the ranked structure of the $Q_{n-1}$ and the existence of the perfect matching between the two copies of $Q_{n-1}$ are crucial. The less nice structure of $\cB_d$ blocks one from a straightforward extension of this argument, but in Section \ref{subsec.rotation} we show that $\cB_d$ still possesses a portion of the structure we desire.
	
	Lemma \ref{lem.step1} follows from Lemma \ref{lem.step1'} below. Given $f \in C_q(\cB_d)$, say $uv \in E$ is \textit{ideal} if $(f(N_u), f(N_v))$ is a principal partition. 
	
	\begin{lem}\label{lem.step1'} 
	Let $\bbf$ be a uniformly random coloring chosen from $C_q(\cB_d)$.
	For any $e \in E$,
		\beq{eq.lem3.1}\pr_\bbf(\mbox{$e$ is not ideal})=2^{-\gO(d)}.\enq
	\end{lem}

	\begin{proof}[Derivation of Lemma \ref{lem.step1} from Lemma \ref{lem.step1'}]

		Given $f$, say a path is \textit{ideal} if all the edges in the path are ideal. Fix any $u \in V$, and write $K_u=K_u(\bbf)$ for the set of vertices that can be reached from $u$ via an ideal path. We may assume $u \in \cL_{d-1}$ by symmetry. 
For $v \ne u$, let $Q^*_{uv}$ be the event $\{v \in K_u\}$. Observe that
			\[\pr(Q^*_{uv}) \ge 1-2^{-\gO(d)} \mbox{ for any $v \ne u$;}\]
\nin indeed, with $P$ a shortest path in $\cB_d$ from $u$ to $v$ (so has length $\le n$), by \Cref{lem.step1'} we have
			\[\pr(Q^*_{uv}) \ge \pr(\mbox{$P$ ideal})\ge 1-n2^{-\gO(d)}=1-2^{-\gO(d)}.\]

		By the above observation, we have $\mathbb E| u \cup K_u|\ge (1-2^{-\gO(d)})|V|$, so by Markov's Inequality, there is a constant $c>0$ such that
		\[\pr(|u \cup K_u|<|V|(1-2^{-cd}))=2^{-\gO(d)}.\]
Note that, by the definition of $K_u$, $f$ agrees with $(f(N_u), \cC \setminus f(N_u))$ at all the vertices in $u \cup K_u$, and $(f(N_u), \cC \setminus f(N_u))$ is principal unless $K_u =\emptyset$. Now Lemma \ref{lem.step1} follows.\end{proof}

\subsection{Structure of $\cB_d$}\label{subsec.rotation}
	
	We first define "levels" of vertices in $\cB_d$. 
		To make the notion more intuitive, we begin by ``rotating" $\cB_d$: for each vertex $x\in V$, let $$x^\rot := x + (\underbrace{1,\ \dots,\  1}_{\text{$d-1$ entries}}, \underbrace{0,\ 0,\ \dots,\ 0}_{\text{ $d$ entries}})$$ 
	(where addition is taken modulo 2). Let $V^\rot=\{x^\rot:x \in V\}$. Notice that this rotation preserves Hamming distance between pairs of vertices; in particular, vertices $x$ and $y$ are adjacent if and only if $x^{\rot}$ and $y^{\rot}$ differ in exactly one entry. Therefore, $\cB_d$ and $\cB_d^\rot:=Q_n[V^\rot]$ are isomorphic, so it suffices to prove \Cref{lem.step1'} for $\cB_d^\rot$ (and we will do so).

The $k^\text{th}$ level of $\cB_d^\rot$ is now defined as in \cite{EG}:	
	\[
		L_k := \left\{x\in V^\rot\ :\   \sum_{i=0}^{n-1}x_i = k\right\}.
	\]
Observe that the levels of $\cB_d^\rot$ ranges from 0 to $n-1$.
To obtain a decomposition of $\cB_d^\rot$ similar to the decomposition of $Q_n=Q_{n-1} \square K_2$, consider the partition $(V_0,V_1)$ of $V^\rot$ where $V_0=\{x \in V^\rot:x_n=0\}$ and $V_1=\{x \in V^\rot:x_n=1\}$. Observe that $|V_0|=|V_1|=N/2$, and $\nabla(V_1,V_2)$ is a (not necessarily perfect) matching (since $\cB_d^\rot$ is a subgraph of $Q_n=Q_{n-1} \square K_2$). Moreover,
\[\mbox{there is a perfect matching between $L_k \cap V_0$ and $L_k \cap V_1$ if $k$ is even,}\]
and
\[\mbox{$\nabla(L_k \cap V_0,L_k \cap V_1)=\emptyset$ if $k$ is odd}.\] 
Indeed, for any $v \in L_k \cap V_0$, 
\beq{def.v prime} v'=v + (0,0,\dots, 0,1)\enq
 is in $L_k \cap V_1$ (we call $v'$ \textit{the mate of $v$} in this case) iff $k$ is even. We denote $V^*=\{x \in L_k \cap V_0:k \text{ even}\}$, and note for future reference that 
		\begin{equation}\label{eq:VStarSize}
			|V^*| = \tbinom{n-1}{d-1} = \tfrac{d}{n}\tbinom{n}{d} = \tfrac{d}{2n}N.
		\end{equation}

\nin For $v \in V^*$, let \[M_v:=N_v \setminus \{v'\} (\sub V_0),\] and define $M_{v'}$ similarly. Finally, for each $v \in V^*$ with the Hamming weight $|v|\ge 4$, we associate 
\beq{def.w} \mbox{a $w = w(v) \in V^*$ with $|w|=|v|-4$ and $w$ being connected to $v$ by a path of length 4 in $\cB_d^\rot[V_0]$.}\enq
This is possible because, in fact, for any $v \in L_k \cap V_0$ and $\underline 0 \in V_0$, their preimages in $\cB_d$, i.e., $v^\rot$ and $\underline 0^\rot$, are connected by a path of length $|v^\rot +\underline 0^\rot|=|v|=k$ in $\cB_d$.

\subsection{Proof of \Cref{lem.step1'}}\label{subsec.pf of step1} In this section, we use $\bar f$ for $f(\cdot)$ to shorten the notation. Again, the argument in this section, which builds upon the decomposition of $\cB_d^\rot$ into $V_0 \cup V_1$ discussed in \Cref{subsec.rotation}, closely follows \cite{Kahn,EG}, and we try to be brief while pointing out differences that is worth noting.

Let $v\in V^*$ (so $v\in L_i$ for some even $i$). For $v'$ and $w=w(v)$ defined in \eqref{def.v prime} and \eqref{def.w}, observe that
	\begin{equation}\label{eq:sub1}
	M_v\cup M_{v'} \subseteq L_{i-1}\cup L_{i+1},
	\end{equation}
and	
	\begin{equation}\label{eq:sub2}
	N_w \subseteq L_{i-5}\cup L_{i-4}\cup L_{i-3}.
	\end{equation}

	\nin Following \cite{Kahn}, we define the following order on the indices of the levels: $1\prec 0\prec 3\prec 2\prec 5\prec 4\cdots.$ To turn this into a partial order on $V$, we say that $x\prec y$ if $i\prec j$, where $x\in L_i$ and $y\in L_j$.
	
	With this partial order, we have $\{i-1,i+1\}\prec i$ and $\{i-5, i-4, i-3\} \prec \{i-1 ,i+1\}$ for even indices $i$. Thus \eqref{eq:sub1} and \eqref{eq:sub2} translate into the following:	
	\begin{equation}\label{eq:order1}
		N_w\prec M_v\cup M_{v'} \prec v,v' \quad \text{for each $v \in V^*$}
	\end{equation}
	(where we write $X\prec Y$ if $x\prec y$ for \emph{all} $x\in X$ and $y\in Y$).

We first point out that for a uniformly random $q$-coloring $\bbf$, by \eqref{entropyprop1} we have the trivial bound:	
	\beq{eq:trivH} H(\bbf)\ge N\log	\left(\hq\right).\enq 		
\nin	To complete the proof of \Cref{lem.step1'}, we wish to give an upper bound on $H(\bbf)$ that would be smaller than the lower bound \eqref{eq:trivH} if $\pr(\mbox{$e$ is not ideal})$ is large, allowing us to conclude that $\pr(\mbox{$e$ is not ideal})=2^{-\gO(d)}$.

We may assume that $e=vv'$ for some $v \in V^*$ by symmetry. 
Observe that
		\beq{lem:epsBreakdown} 
		\eps:= 1-\hspace{-.3ex}\sum_{\text{$(A,B)$ principal}} \pr(\bar\bbf_{M_v}=A,\   \bar\bbf_{M_{v'}}=B)\ge 1-\hspace{-.3ex}\sum_{\text{$(A,B)$ principal}} \pr(\bar\bbf_{N_v}=A,\   \bar\bbf_{N_{v'}}=B)=\pr(\mbox{$e$ is not ideal});\enq
therefore to prove \Cref{lem.step1'}, it will suffice to show that $\eps = 2^{-\Omega(d)}$.

	We are now ready to compute an upper bound for $H(\bbf)$. 
Note that the family $\{M_v\cup M_{v'}: v\in V^*\}$ together with $d$ copies of each edge $vv'$ with $v\in V^*$ covers each vertex of $V$ exactly $d$ times, so by \Cref{lem:Sh}, we have 
	\[H(\bbf)\le \frac{1}{d} \left(d \sum_{v \in V^*} H\big(\bbf_{v,v'}\big|(\bbf_i:i\prec v,v')\big)+\sum_{v \in V^*}H\big(\bbf_{M_v,M_{v'}}\big|(\bbf_i:i\prec M_v,M_{v'})\big)\right),\]
which can be relaxed to
	\begin{equation}\label{18} H(\bbf)\le  \sum_{v \in V^*} H(\bbf_{v,v'}|\bbf_{M_v,M_{v'}})+\frac{1}{d}\sum_{v \in V^*}H(\bbf_{M_v,M_{v'}}|\bbf_{N_w})
	\end{equation}
using \eqref{eq:order1}.

		For the first term on the rhs of \eqref{18}, 
	\beq{eq:term1} \begin{split}
		H(\bbf_{v,v'}|\bbf_{M_v,M_{v'}})
		&\stackrel{\eqref{entropyprop4}}{\le} H(\bbf_{v,v'}|\bar\bbf_{M_v},\bar\bbf_{M_{v'}})		\\
		&\stackrel{\eqref{condent}}{=} \sum_{A_0,A_1}  H\big(\bbf_{v,v'}\,\big|\ \bar\bbf _{M_v}=A_0,\, \bar\bbf_{M_v'}=A_1\big)\cdot \pr(\bar\bbf_{M_v}=A_0,\, \bar\bbf_{M_v'}=A_1),\\
	\end{split}\enq
where the sum is taken over \emph{all} possible $A_0,A_1\subseteq \cQ$. Note that
\beq{eq:term1'} H\big(\bbf_{v,v'}\,\big|\ \bar\bbf _{M_v}=A_0,\, \bar\bbf_{M_v'}=A_1\big) \stackrel{\eqref{entropyprop1}}{\le} \log\big[(q-|A_0|)(q-|A_1|)-|\cQ \setminus(A_0\cup A_1)|\big]\enq
since $v \sim v'$ so they cannot take the same color.
	
	Next we treat the second term on the rhs of \eqref{18}. For $v\in V^*$ with $|v|\le 3$ (the number of such vertices is at most ${n \choose \le 3}=O(d^3)$), the vertex $w=w(v)$ is not defined, and we simply apply the naive bound
	\[H(\bbf_{M_v,M_{v'}}|\bbf_{N_w})\ =\ H(\bbf_{M_v,M_{v'}})\ \le\ 2(d-1)\log q.\]
	
	For $v\in V^*$ with $|v|\ge 4$, we bound the conditional entropy in two pieces, corresponding to (i) the choice of color sets for $M_v$ and $M_{v'}$, and (ii) the assignment of colors to specific vertices once the sets of colors are determined. Specifically, we have
	\begin{align}
		H(\bbf_{M_v,M_{v'}}|\bbf_{N_w})& \stackrel{\eqref{entropyprop4}}{\le} H(\bbf_{M_v,M_{v'}}\big|\bar\bbf_{N_w})\nonumber\\
		&\stackrel{\eqref{entropyprop5}}{=}H(\bbf_{M_v,M_{v'}}, \bar\bbf_{M_v}, \bar\bbf_{M_{v'}}\big|\bar \bbf_{N_w})\nonumber\\
		&\stackrel{\eqref{entropyprop1.5},\eqref{entropyprop2}}{\le} H(\bar\bbf_{M_v}, \bar\bbf_{M_{v'}}\big|\bar \bbf_{N_w})+H(\bbf_{M_v,M_{v'}}\big| \bar\bbf_{M_v}, \bar\bbf_{M_{v'}}).\label{21}\
	\end{align}
	
The second term in \eqref{21} is
	\begin{align}
		H(\bbf_{M_v,M_{v'}}\big| \bar\bbf_{M_v}, \bar\bbf_{M_{v'}}) 
		& \stackrel{\eqref{condent}}{=}\sum_{A_0,A_1} 
		H(\bbf_{M_v,M_{v'}}\big|\bar \bbf_{M_v}=A_0,\,\bar \bbf_{M_v'}=A_1)
		\cdot\pr(\bar \bbf_{M_v}=A_0,\,\bar \bbf_{M_v'}=A_1)\nonumber\\
		& \stackrel{\eqref{entropyprop1}}{\le} \sum_{A_0,A_1}\log(|A_0|^{d-1}|A_1|^{d-1})\cdot \pr(\bar \bbf_{M_v}=A_0,\,\bar \bbf_{M_v'}=A_1).\label{22}
\end{align}

\begin{remark}\label{stark} There is no simple improvement possible to this trivial bound in \eqref{22} since there are no edges between $M_v$ and $M_{v'}$. This is in \textit{stark contrast} to the situation in $Q_n$, where there is a perfect matching between $M_v$ and $M_{v'}$, allowing an entropy savings whenever $A_0\cap A_1$ is nonempty (since we cannot choose the same color for both vertices in any of the matching edges). Ultimately, the loss of this small entropy savings is the reason why we are unable to establish \Cref{lem.step1'} for odd $q$.
\end{remark}

Now, plugging all our bounds (\eqref{eq:term1}, \eqref{eq:term1'}, \eqref{21}, and \eqref{22}) into \eqref{18} and simplifying slightly, we have
	
	\begin{align}
		H(\bbf)\ \ &\le \ \ O(d^3)\  + \ \frac{|V^*|}{d}H(\bar\bbf_{M_v}, \bar\bbf_{M_{v'}}\big|\bar \bbf_{N_w}) 
		\ \ +\ \ |V^*|\hspace{-.6ex}\sum_{A_0,A_1} 	\log(h_{A_0,A_1})\cdot \pr(\bar \bbf_{M_v}=A_0,\,\bar \bbf_{M_v'}=A_1),\label{23}
	\end{align}
	where $		h_{A_0,A_1}:= \Big((q-|A_0|)(q-|A_1|)-|\cQ \setminus(A_0\cup A_1)|\Big)\cdot \big(|A_0||A_1|\big)^{1-1/d}.$

		To bound the final sum, we bound $h_{A_0,A_1}$ as follows: first, if $(A_0,A_1)$ is principal, then (recalling $q$ is even)
	\begin{align*}
		h_{A_0,A_1} 
		= \Big(\left(\frac{q}{2}\right)^2-0\Big)\cdot \left(\left(\frac{q}{2}\right)^2\right)^{1-1/d} 
		= \left(\frac{q}{2}\right)^{4-2/d}.
	\end{align*}
	If $(A_0,A_1)$ is not principal, then $h_{A_0,A_1}$ is substantially smaller; to see this, we begin by writing
	\[h_{A_0,A_1}\le \Big((q-|A_0|)(q-|A_1|)-|\cQ \setminus(A_0\cup A_1)|\Big)\cdot \big(|A_0||A_1|\big)\]
	(this is true for all $A_0,A_1$, but it will simplify our computations slightly in the non-principal case). Now, if
	 $(A_0,A_1)$ is not principal but $A_0\cup A_1 = \cQ$, then $|A_0|\neq q/2$ or $|A_1|\neq q/2$. Without loss of generality, say that $|A_0|\neq q/2$. So $(q-|A_0|)|A_0|\leq  \left(q/2\right)^2-\Theta(1)$, and $(q-|A_1|)|A_1|\leq \left(q/2\right)^2 $, giving
	\[h_{A_0,A_1}\le \big((q-|A_0|)|A_0|\big)\cdot\big((q-|A_1|)|A_1|\big)\le \left(\frac{q}{2}\right)^4-\Theta(1).\]
	(Notice that the same statement would \emph{not} be true for odd $q$, as we could take a non-principal partition $(A_0,A_1)$ with $|A_0|, |A_1| = \lceil q/2\rceil$, giving the same bound as in the principal case.)
	And if $(A_0,A_1)$ is not principal and $A_0\cup A_1 \neq \cQ$, then 
	\[h_{A_0,A_1}\le \Big((q-|A_0|)(q-|A_1|)-\Theta(1)\Big)\cdot \big(|A_0||A_1|\big) \le \left(\frac{q}{2}\right)^4-\Theta(1)\]
	in this case as well.

	So in either case, if $(A_0,A_1)$ is not principal, then $h_{A_0,A_1}\le \left(q/2\right)^4-\Theta(1)$, which we will rewrite as $h_{A_0,A_1}\le\left(q/2\right)^{4-2/d}-\gd$ where $\gd=\Theta(1)$ (since $q$ is a constant and $d\rightarrow\infty$). 
Recalling the definition of $\eps$ from \eqref{lem:epsBreakdown}, we may bound the sum in \eqref{23} as follows:
	\[\begin{split}\sum_{A_0,A_1} 	\log(h_{A_0,A_1})\cdot \pr(\bar \bbf_{M_v}=A_0,\,\bar \bbf_{M_v'}=A_1)
		&\ \le \eps\log\left(\left(\frac{q}{2}\right)^{4-2/d}-\gd\right)+(1-\eps)\log \left(\left(\frac{q}{2}\right)^{4-2/d}\right) \\
		&\ =\log\left(\left(\frac{q}{2}\right)^{4-2/d}\right) -\eps \log\left(1-\gd/\left(\frac{q}{2}\right)^{4-2/d}\right)^{-1}\\
		&\ =\log\left(\left(\frac{q}{2}\right)^{4-2/d}\right) -\eps \cdot \Theta(1).
	\end{split}\]	
	Inserting this into \eqref{23}, we have
	\beq{27} H(\bbf) \le O(d^3)+\frac{|V^*|}{d}H(\bar\bbf_{M_v}, \bar\bbf_{M_{v'}}\big|\bar \bbf_{N_w}) +|V^*|\left(\log\left(\left(\frac{q}{2}\right)^{4-2/d}\right) -\eps\cdot \Theta(1)\right).\enq

On the other hand, \eqref{eq:VStarSize} gives $|V^*|\log\left(\left(q/2\right)^{4-2/d}\right) = N\log\left(q/2\right)$. Then combining \eqref{eq:trivH} and \eqref{27}, and solving for $\eps$, we obtain
	\begin{equation}\label{eq:epsSolved}
		\eps\ \le\  O(d^3/N)+O\left(1/d\right)\cdot H(\bar\bbf_{M_v}, \bar\bbf_{M_{v'}}\big|\bar \bbf_{N_w}).
	\end{equation}

	Thus all that remains in order to bound $\eps$ is to analyze the entropy term in \eqref{eq:epsSolved}. Ignoring the conditioning, a naive upper bound is
	\begin{equation}\label{eq:trivHC}
	H(\bar\bbf_{M_v}, \bar\bbf_{M_{v'}}\big|\bar \bbf_{N_w}) \stackrel{\eqref{entropyprop1}}{\le} 2q.
	\end{equation}
Substituting into \eqref{eq:epsSolved}, together with recalling that $N = 2^{d(1-o(1))}$ gives
\beq{eps.o(1)} \eps =O(1/d)~ (=o(1)).\enq 
	
	We will strengthen our bound on $H(\bar\bbf_{M_v}, \bar\bbf_{M_{v'}}\big|\bar \bbf_{N_w})$ via the following key lemma\footnote{This lemma is based on Lemma~4.2 in Engbers and Galvin's paper \cite{EG}, and the proof is very similar. However, there is a small subtle error in the original proof of Lemma~4.2, which can be resolved by slight modifications to the argument, analogous to those made here.} (and the fact that $\eps$ is small). 
	
	\begin{lem}\label{Lem4.2}
		For any principal $(A,B)$,
		\beq{Lem4.2.first} \pr\left(\bar\bbf_{M_{v}}=A,\  \bar\bbf_{M_{v'}}=B\, \big|\, \bar\bbf_{N_{w}}=A\right)\ge 1-\frac{6\eps}{\pr\left(\bar\bbf_{N_{w}}=A\right)},\enq
		and also
		\beq{Lem4.2.second}\sum_{\text{A not principal}} \pr(\bar\bbf_{N_{w}}=A) \le \eps.\enq
	\end{lem}
	
	\begin{proof}
	Let $w, w_1, w_2, w_3, v$ be a path of length $4$ from $w$ to $v$ in $\cB_d^\rot$. 
Write $Q_{v,A}$ and $R_{v,A}$ for the events $\{\bar \bbf_{N_v}=A\}$ and $\{\bar \bbf_{M_v}=A\}$, respectively. (The event $R_{v,A}$ will be defined only for $v \in V^*$ and its mate, so we don't have to worry about the definition of $M_v$ for other $v$'s.) 
	For two events $P,Q$ we use the shorthand notation $PQ$ for $P \wedge Q$, and $\bar P$ for the complementary event of $P$. We also use $Q_{vv',(A,B)}$ for $Q_{v,A}Q_{v',B}$, and similarly for $R_{vv',(A,B)}$.
	
	With this notation, to prove the first statement of the lemma, we wish to bound the probability
	\[
		\pr\left(R_{vv',(A,B)}\big|\, Q_{w,A}\right)\ =\ 1-\pr\left(\bar R_{vv',(A,B)},\big|\, Q_{w,A}\right) \ =\ 1-\frac{\pr\left(\bar R_{vv',(A,B)} Q_{w,A}\right)}{\pr\left(Q_{w,A}\right)}.
	\]
To bound the probability in the numerator, we begin by observing that
	\beq{bd.3}\bar R_{vv',(A,B)} Q_{w,A} \sub \bar R_{vv',(A,B)} Q_{vv',(A,B)} \vee \bar Q_{vv',(A,B)} Q_{w,A}.\enq
	This does not rely on any special facts about these events; in general, $PQ\ =\ PQ \wedge (R\vee \bar R)\ \sub\ PR \vee \bar RQ$.
	
Moreover, we have
	\beq{bd.1} \pr(\bar R_{vv',(A,B)} Q_{vv',(A,B)}) \le \eps.\enq
	To see this, notice that the event $\bar R_{vv',(A,B)} Q_{vv',(A,B)}$ is equivalent to saying that $(\bar\bbf_{M_{v}},\bar\bbf_{M_{v'}}) \neq (A,B)$, but $(\bar\bbf_{N_{v}},\bar\bbf_{N_{v'}}) = (A,B)$. Since $\bar\bbf_{M_{v}}\sub \bar\bbf_{N_{v}}$ and $\bar\bbf_{M_{v'}}\sub \bar\bbf_{N_{v'}}$, this means that there must be at least one color missing from $\bar\bbf_{M_{v}}\cup\bar\bbf_{M_{v'}}$; thus this event implies that $(\bar\bbf_{M_{v}},\bar\bbf_{M_{v'}})$ is not principal. Recalling that $\eps$ is simply the probability that $(\bar\bbf_{M_{v}},\bar\bbf_{M_{v'}})$ is not principal, this gives \eqref{bd.1}.

	Now to bound the probability of the remaining event, $\bar Q_{vv',(A,B)}Q_{w,A}$, we observe
	\[\bar Q_{vv',(A,B)}Q_{w,A} \sub Q_{w,A}\bar Q_{w_1,B} \vee Q_{w_1,B}\bar Q_{w_2,A} \vee Q_{w_2,A}\bar Q_{w_3,B} \vee Q_{w_3,B}\bar Q_{v,A} \vee Q_{v,A} \bar Q_{v',B},\]
	and 
	\beq{bd.2} \mbox{each of the 5 events on the rhs occurs with probability less than $\eps$}\enq since  
$Q_{w_1,B}\bar Q_{w_2,A}$ implies that $w_1w_2$ is not ideal, etc. Therefore,
	\[\begin{split}
		\pr(\{\bar\bbf_{M_v}=A, \bar\bbf_{M_{v'}}=B\}^c \mid \bar\bbf_{N_w}=A)&=\frac{\pr(\bar R_{vv',(A,B)}Q_{w,A})}{\pr(\bar\bbf_{N_w}=A)} \stackrel{\eqref{bd.3},\eqref{bd.1},\eqref{bd.2}}{\le} \frac{6\eps}{\pr(\bar\bbf_{N_w}=A)}
	\end{split}\]
which concludes \eqref{Lem4.2.first}.
	Also, $\pr(\bar\bbf_{N_{w}}=A) \ge \pr(\bar\bbf_{N_{w}}=A,\, \bar\bbf_{N_{w'}}=\cC \setminus A)$ implies
	\[\begin{split}
		\sum_{\text{$A$ principal}} \pr(\bar\bbf_{N_{w}}=A) &\ge \sum_{\text{$A$ principal}}\pr(\bar\bbf_{N_{w}}=A,\, \bar\bbf_{N_{w'}}=\cC \setminus A)\\
		&=\sum_{\text{$(A,B)$ principal}} \pr(\bar\bbf_{N_{w}}=A,\, \bar \bbf_{N_{w'}}=B)\\
		&= \pr(\text{$e$ is ideal}) \geq 1-\eps
	\end{split}\]
	where the last inequality follows from \eqref{lem:epsBreakdown}.
	\end{proof}

\iffalse
\begin{remark}
For Galvin's proof (TBA)
\end{remark}

\fi

The rest of the proof is just a special case of the proof of \cite[Theorem 1.4]{EG} so we will try to be brief.	
The entropy term in \eqref{eq:epsSolved} can be splited as
	\beq{28.5}\begin{split}
		H(\bar\bbf_{M_v}, \bar\bbf_{M_{v'}}\big|\bar \bbf_{N_w})\ &\stackrel{\eqref{condent}}{=}\sum_{\text{$A$ principal}} H(\bar\bbf_{M_v}, \bar\bbf_{M_{v'}}\big|\bar \bbf_{N_w}=A)\cdot \pr(\bar \bbf_{N_w}=A)\\
		&+\sum_{\text{$A$ not principal}} H(\bar\bbf_{M_v}, \bar\bbf_{M_{v'}}\big|\bar \bbf_{N_w}=A)\cdot \pr(\bar \bbf_{N_w}=A).
	\end{split}\enq	
The second sum in \eqref{28.5} is
	\beq{29} \sum_{\text{$A$ not principal}} H(\bar\bbf_{M_v}, \bar\bbf_{M_{v'}}\big|\bar \bbf_{N_w}=A)\cdot \pr(\bar \bbf_{N_w}=A) \stackrel{\eqref{Lem4.2.second},\eqref{entropyprop1},\eqref{entropyprop1.5}}{\le} 2q\eps;\enq
for the first sum in \eqref{28.5}, note that, by symmetry, $\pr(\bar \bbf_{N_w}=A)$ is equal for each principal $A \sub \cQ$, so by \eqref{Lem4.2.second} and \eqref{eps.o(1)}, we have
	\beq{lb}\pr(\bar \bbf_{N_w}=A)=\gO(1).\enq
Next, for each principal $A$,
	\[\begin{split}
		H(\bar\bbf_{M_v}, \bar\bbf_{M_{v'}}\big|\bar \bbf_{N_w}=A) 
		\stackrel{\eqref{condent}}{=}\sum_{A_0, A_1}-\pr(\bar\bbf_{M_{v}}=A_0,\  \bar\bbf_{M_{v'}}=A_1\big| \bar\bbf_{N_{w}}=A)\cdot \log\big[\pr(\bar\bbf_{M_{v}}=A_0,\  \bar\bbf_{M_{v'}}=A_1\big| \bar\bbf_{N_{w}}=A)\big],
	\end{split}\]
and we use the fact that  $-p\log p\leq H(p)$ for any constant $p \in (0,1)$ to obtain 
	\beq{30} H(\bar\bbf_{M_v}, \bar\bbf_{M_{v'}}\big|\bar \bbf_{N_w}=A) \le \sum_{A_0, A_1} H\Big(\pr(\bar\bbf_{M_{v}}=A_0,\  \bar\bbf_{M_{v'}}=A_1\, \big|\, \bar\bbf_{N_{w}}=A)\Big).\enq
If $(A_0,A_1)=(A,\cQ \setminus A)$, then \Cref{Lem4.2} gives
	\[\pr(\bar\bbf_{M_{v}}=A_0,\  \bar\bbf_{M_{v'}}=A_1\, \big|\, \bar\bbf_{N_{w}}=A)
	\ge 1- \frac{6\eps}{\pr(\bar\bbf_{N_{w}}=A)}\stackrel{\eqref{eps.o(1)},\eqref{lb}}{\ge} 1/2.\]
On the other hand, if $(A_0, A_1) \ne (A, \cQ \setminus A)$ then by \Cref{Lem4.2}
	\[\pr(\bar\bbf_{M_{v}}=A_0,\  \bar\bbf_{M_{v'}}=A_1\, \big|\, \bar\bbf_{N_{w}}=A)
	\le \frac{6\eps}{\pr(\bar\bbf_{N_{w}}=A)}\le 1/2.\]
Since $H(p)$ is increasing for $p\leq 1/2$, and is symmetric about $1/2$, we may bound each of the entropy terms in the rhs of \eqref{30} by $H(6\eps/\pr(\bar\bbf_{N_{w}}=A))$. Thus we may bound the first sum in \eqref{28.5} as follows:
	\beq{32}\begin{split}
		\sum_{\text{$A$ principal}}H(\bar\bbf_{M_v}, \bar\bbf_{M_{v'}}\big|\bar \bbf_{N_w}=A)\cdot \pr(\bar \bbf_{N_w}=A)\ 
		&\ \le 2^{2q} \sum_{\text{$A$ principal}}H\left(\frac{6\eps}{\pr(\bar \bbf_{N_w}=A)}\right)\cdot\pr(\bar \bbf_{N_w}=A)\\
		&=  O(1)\cdot H\left(\frac{6\eps}{\pr(\bar \bbf_{N_w}=A)}\right)\\	
		& \le 	C \eps \log(C/\eps)
	\end{split}\enq
for some constant $C$, where the last inequality uses the fact that $H(x)\le 2x \log (1/x)$ for $x \le 1/2$ (and \eqref{lb}).
	
	Finally, we combine \eqref{28.5}, \eqref{29}, and \eqref{32} to obtain
\[		H(\bar\bbf_{M_v}, \bar\bbf_{M_{v'}}\big|\bar \bbf_{N_w})
		\le 2q\eps + C \eps \log(C/\eps),\]
which, together with \eqref{eq:epsSolved}, gives
\[	\eps\ \le\  O(d^3/N)+O\left(1/d\right)\cdot \big[2q\eps + C \eps \log(C/\eps)\big].\]
Then as $N= 2^{\Theta(d)}$, we solve to obtain $\eps =2^{-\gO(d)}$, completing the proof of \Cref{Lem4.2}.

\section{Proof of Lemma \ref{lem.step2}}\label{sec.step2}

\subsection{Warming-up}\label{subsec.entropy function} Before we start in earnest, we recall (e.g. from \cite{KPq}) how the entropy function gives an upper bound on $c_q(G)$ for any $d$-regular bipartite graph $G$ (with, say, a bipartition $\cD \cup \cU$). The argument in this section will be used in a refined form in the proof in \Cref{subsec.step2.pf}. For $\bbf \in C_q(G)$ chosen uniformly at random, we have
\beq{case1.f} \begin{split}\log c_q(G) \stackrel{\eqref{entropyprop1}}{=} H(\bbf)&\stackrel{\eqref{entropyprop2}}{=}H(\bbf_\cD)+H(\bbf_\cU|\bbf_\cD)\le \sum_{u \in \cU}\left[\frac1dH(\bbf_{N_u})+H(\bbf_u|\bbf_{N_u})\right];
\end{split}\enq
here the inequality uses
Lemma~\ref{lem:Sh} with
\beq{alphaS}
\alpha_S=\left\{\begin{array}{ll}
1/d&\mbox{if $S=N_u$ for some $u\in \cU$;}\\
0&\mbox{otherwise}
\end{array}\right.
\enq
for the first term, and \eqref{entropyprop3} for the second term.

Notice that we can always give the following simple upper bound on the terms in \eqref{case1.f}; this bound is too weak to be useful in general, but in the proof that follows, we will use it on some small sets of vertices.

\begin{prop}\label{entropy.bm}
For $\bbf$ drawn from \emph{any} probability distribution on $C_q(G)$ and $u \in V$, 
\beq{eq.Tu.def}(T(u):=) \; \frac1dH(\bbf_{N_u})+H(\bbf_u|\bbf_{N_u}) \le \log (q/2)^2 +O(1/d).\enq
\end{prop}

\begin{proof}
Recall that $\bbf(N_u)=\{\bbf_v: v\in N_u\}$, i.e., the set of colors used on $\bbf_{N_u}$.
Clearly, $\bbf(N_u)$ is determined by $\bbf_{N_u}$, and then we have
\begin{eqnarray}
T(u)
&\stackrel{\eqref{entropyprop2}}{=}&\tfrac{1}{d}[H(\bbf(N_u))+H(\bbf_{N_u}|\bbf(N_u))]+H(\bbf_u|\bbf_{N_u})\nonumber\\
&\stackrel{\eqref{entropyprop4}}{\le}&\tfrac{1}{d}H(\bbf(N_u))+\tfrac{1}{d}H(\bbf_{N_u}|\bbf(N_u))+H(\bbf_u|\bbf(N_u)).\label{def_tu}
\end{eqnarray}
Note that for each possible value $C$ of $f(N_u)$,
\begin{eqnarray*}
H(\bbf_u|\bbf(N_u)=C) &\leq & \log(q-|C|),\\
H(\bbf_{N_u}|\bbf(N_u)=C) &\leq & d\log|C|.
\end{eqnarray*}
Since $\log x+\log (q-x)\leq \log \qq$ for $x \in \mathbb N$,
the last two terms in \eqref{def_tu} are bounded by
\[
\sum_C\pr(\bbf(N_u)=C)\left[\frac{1}{d}H(\bbf_{N_u}|\bbf(N_u)=C)+H(\bbf_u|\bbf(N_u)=C)\right]\leq \log \qq,
\]
yielding the proposition (since $H(\bbf(N_u))=O(1)$).
\end{proof}

\subsection{Proof overview} 
In Step 1 (\Cref{lem.step1}), we showed that the vast majority of colorings in $C_q(\cB_d)$ admit a ground state with exponentially small flaws. Our goal in this section is to bound the number of such colorings that disagree with their ground state exactly at a particular set $X$ (ranging over all 2-linked sets $X$ with moderately-sized neighborhoods). More precisely, we wish to bound the quantity
\beq{quantity}
\sum_{X \in \cH(g)} \chi_{A,B}(X),
\enq
for all $g$ and $(A,B)$ as in \Cref{lem.step2}.
As is typical as in this line of work, directly bounding the number of colorings $f$ in $\chi_{A,B}(X)$ is a difficult task. The main idea of the current proof, which is inspired by \cite{KPq, misqn}, is to bound the number of colorings $f$ that admit a given pair $(F,S)$ as a $\psi$-approximating pair of $X_{A,B}(f)$. (Note: \Cref{sec.prelim} has definitions and background on $\psi$-approximating pairs.) 

Roughly speaking, for a given pair $(F,S)$, we bound the entropy of a random coloring $\mathbf{f}$ corresponding to $(F,S)$, breaking our analysis into three cases. In each case, we use the chain rule of conditional entropy (as in \eqref{case1.f} above) to first ``expose" part of the coloring, then use this to find some ``entropy savings" as we color the remaining vertices. The exact order in which we expose the vertices of the coloring will be different in each case; very roughly, the three cases will depend on the difference in size between $F$ and $S$ and on the amount of error obtained in approximating the flaw $X$ by the pair $(F,S)$. In each case, $(F,S)$ provides a different type of "resources" that we can use to obtain entropy savings. 

We emphasize that the containers $(F,S)$ are crucially used to lead entropy savings; a ``vanilla" entropy approach (as in \Cref{subsec.entropy function}) gives a much weaker bound, and a careful combination of entropy and containers is key in this step.

We also remark that, as pointed out in \Cref{sec.intro}, the proof in this section can easily be extended to odd $q$.

\subsection{Notation and proof setup}

Throughout this section, we fix an integer $g$ and a principal partition $(A,B)$ as in \Cref{lem.step2}. In order to use approximating pairs, we must specify the parameter $\psi$; let $\psi$ be an arbitrary number with
\beq{psi.bd} \log^3 d \ll \psi \ll d/\log^2 d,\enq 
and let $\cU(g)$ be the corresponding set of pairs $(F,S)$ guaranteed by Lemma~\ref{lem:container}. 

For a pair $(F, S)\in \cU(g)$, we define 
\[
\mathcal{I}_{F, S}:=\{X\in\cH(g): \mbox{$(F,S)$ is a $\psi$-approximating pair of $X$}\}.
\]
Observe that with this notation, \eqref{quantity} can be broken down according to approximating pairs as follows:
\beq{qtyBreakdown}
	\sum_{X \in \cH(g)} \chi_{A,B}(X) = \sum_{(F,S)\in\cU(g)}\ \sum_{X\in \cI_{F,S}} \chi_{A,B}(X)
\enq
Before bounding this sum (the goal of this section), we will break it down one step further  according to the size of a particular set related to each $X$ (yielding \eqref{qtyFinalBdown} below). To do so, we will need some additional notation; we also take this opportunity to introduce a variety of notational conventions for this section.

Given a pair $(F,S)\in \cU(g)$ and a set $X\in \cI_{F,S}$, we will adopt the following:

\begin{itemize}
	\item We use $\cD$ and $\cU$ for $\cL_{d-1}$ and $\cL_d$ (for notational simplicity) 
	\item $X_1:=X \cap \cD$, \ $S_1:=S \cap \cD$, and $F_1:=F \cap \cU$
	\item $X_2:=X \cap \cU$, \ $S_2:=S \cap \cU$, and $F_2:=F \cap \cD$
	\item $x_1:=|X_1|$, \ $s_1:=|S_1|$, and $f_1:=|F_1|$
	\item $x_2:=|X_2|$, \ $s_2:=|S_2|$, and $f_2:=|F_2|$
	\item $\eps:=1/\log^{1.5} d$ 
	\item $\eps':=1/\log^2 d$
	\item We assume that $|N(X_1)| \ge |N(X_2)|$ (thus $|N(X_1)| \ge g/2$), as the other case may be handled identically by switching the roles of $\cU$ and $\cD$
	\item $g_1:=|N(X_1)|$; note that the previous convention guarantees $g_1 \in [g/2,g]$
\end{itemize} 
Now for any fixed $g_1\in [g/2, g]$, we define the set of colorings
\beq{f.chosen}
\cF_{F, S}(g_1):=\big\{f : \left(\mbox{$f \in \chi_{A,B}(X)$ for some $X\in\mathcal{I}_{F, S}$}\right) \,\wedge\, \left(|N(X_1)|=g_1\right)\big\}. \enq

The key of the proof is to establish an upper bound for each $|\cF_{F, S}(g_1)|$; concretely, we can break down the sum \eqref{qtyBreakdown} as follows

\beq{qtyFinalBdown}\sum_{X\in\cH(g)}\chi_{A, B}(X)\ = \
	\sum_{(F,S)\in\cU(g)}\ \sum_{X\in \cI_{F,S}} \chi_{A,B}(X)\ =\ \sum_{(F, S)\in \cU(g)}\sum_{g_1\in[g/2, g]}2\cdot |\cF_{F, S}(g_1)|
\enq
(note: the factor of 2 above comes from switching the roles of $\cU$ and $\cD$ in the case where $|N(X_1)| \le |N(X_2)|$). The remainder of this section is dedicated to bounding the terms of this sum, and we will divide our analysis into three cases, detailed below.

As a final note, given a coloring $f$, we call a vertex $v$ is \textit{good} if $f$ agrees with $(A,B)$ at $v$ (and \textit{bad} otherwise). 
We will often use the easy fact that for a random coloring $\bbf$ drawn from \textit{any} probability distribution on $C_q(\cB_d)$, and any $v\in V$,
\beq{triv.bd} \mbox{$H(\bbf_v\,|\,v \text{ is good}) \le \log (q/2)$.}
\enq

\subsection{Proof of \Cref{lem.step2}}\label{subsec.step2.pf}

We now bound the terms $|\cF_{F, S}(g_1)|$ in \eqref{qtyFinalBdown}. Our strategies will vary depending on the values of $g_1 - f_1$ and $f_1 - s_1$, and we will give a different bound on $|\cF_{F, S}(g_1)|$ in each of the three following cases:
\begin{itemize}
	\item[\textbf{Case 1.}] $g_1-f_1<\eps g_1$
	\item[\textbf{Case 2.}] $g_1-f_1 \ge \eps g_1$\, and\, $f_1-s_1 < \eps' g_1$
	\item[\textbf{Case 3.}]	$f_1-s_1 \ge \eps' g_1$
\end{itemize}

\mbox{}\\\noindent\textbf{Case 1. $g_1-f_1<\eps g_1$.} 

We first claim that given a pair $(F_1,S_1)$, 
\beq{NX.cost} \mbox{the number of possibilities for $N(X_1)$  is $2^{o(g_1)} ~ (=2^{o(g)})$.} \enq
 Indeed, \eqref{approx1} implies that $N(X_1)$ is a subset of $N(S_1)$, which is a set of size at most $ds_1 \stackrel{\eqref{sf.ub}}{\le} (1+o(1))dg_1$. Since $F_1 \sub N(X_1)$, $N(X_1)$ is fully determined by $N(X_1) \setminus F_1$, which has size $g_1-f_1<\eps g_1$. Therefore, the number of choices for $N(X_1) \setminus F_1$ (and hence, for $N(X_1)$) is at most 
\[{(1+o(1))dg_1 \choose \le \eps g_1}\le \exp_2[O(\eps g_1\log(d/\eps))]
{=}2^{o(g_1)}\]
(where the final equality is from $\eps:=1/\log^{1.5} d$). Note that $N(X_1)$ does not necessarily determine $X_1$; it only gives the \textit{closure} of $X_1$, i.e., \[[X_1]:=\{v \in \cD:N(v) \sub N(X_1)\}.\] 

Now, fix $G_1$ to be any of the $2^{o(g_1)}$ possible choices for $N(X_1)$, and let $\bbf$ be uniformly chosen from among corresponding colorings, that is, from the set
\[\cF_{F, S}(g_1)\,\cap\, \{\mbox{$f \in \chi_{A,B}(X)$ :  $X\in\mathcal{I}_{F, S}$ and $|N(X_1)|=G_1$}\}.\]  

We bound the entropy of $\mathbf{f}$ as follows:
\begin{eqnarray}
H(\bbf)&\stackrel{\eqref{entropyprop2}}{=}& H(\bbf_\cD)+H(\bbf_\cU|\bbf_\cD)\nonumber\\
&\le & \sum_{u \in G_1 \cup S_2}\frac1dH(\bbf_{N_u}) + \sum_{v \in \cD}\left(1-\frac{d_{G_1 \cup S_2}(v)}{d}\right)H(\bbf_v)+\sum_{u\in\cU}H(\bbf_u|\bbf_{N_u})\nonumber\\
&\le & \sum_{u \in G_1 \cup S_2}T(u)+\sum_{u \in \cU \setminus (G_1 \cup S_2)} H(\bbf_u)+\sum_{v \in \cD}\left(1-\frac{d_{G_1 \cup S_2}(v)}{d}\right)H(\bbf_v),\label{case1.bd}
\end{eqnarray}
where the first inequality uses Lemma~\ref{lem:Sh} with $\alpha_S=1/d$ if $S=N_u$ for some $u \in G_1 \cup S_2$, and $\alpha_v=1-d_{G_1 \cup S_2}(v)/d$ for each singleton $v \in \cD$ (so that \eqref{fractiling} is satisfied).
We will bound each term of \eqref{case1.bd} individually.\\

\nin (i) \textit{First term of \eqref{case1.bd}:} We define the set
\begin{equation}\label{xminus}
[X_1]^-=\{u \in V: N(u) \sub [X_1]\},
\end{equation} 
and note that $[X_1]^-\subseteq G_1$. , 
If $u \in [X_1]^- \cup (S_2 \setminus G_1)$, then we use the naive bound in \Cref{entropy.bm} to obtain

\beq{case 1. first term 1} T(u) \le \log \qq +O(1/d).\enq

For the remaining vertices $u\in G_1\cup S_2$ (that is, those vertices in $G_1\setminus [X_1]^-$), we can achieve a stronger bound on $T(u)$:

\begin{prop}\label{prop5.3} 
	For $u \in G_1 \setminus [X_1]^-$               with $d'(u):=d-d_{[X_1]}(u)$,
\[ \frac{1}{d}H(\bbf_{N_u}|\bbf(N_u))+H(\bbf_u|\bbf(N_u))=\log \qq-\gO(d'(u)/d). \]
\end{prop}

\begin{proof}
Note that if a color set $C \sub \cQ$ is not principal, i.e., $c:=|C|\neq q/2$, then
\beq{nonprinc}\begin{split} \frac{1}{d}H(\bbf_{N_u}|\bbf(N_u)=C)+H(\bbf_u|\bbf(N_u)=C)& \le \log c +\log (q-c) =\log \qq-\gO(1).\end{split}\enq
Next, suppose $\bbf(N_u)=C$ for some principal $C$, i.e., $c= q/2$. Note that $u \in G_1 \setminus [X_1]^-$ implies $u \sim X_1$ and $u \sim \cD \setminus X_1$, so in particular $C$ intersects both $A$ and $B$. Let $C' = C \cap B ~ (\ne \emptyset)$ and $c'=|C'| (\le c-1)$. Note that each vertex in $N_u \setminus [X_1]$ is good, so it has at most $c'$ color choices. For vertices in $[X_1]$, we apply the naive upper bound $c~ (=q/2)$ for the number of possible color choices. So we have
\beq{princ}\begin{split} \frac{1}{d}H(\bbf_{N_u}|\bbf(N_u)=C)+H(\bbf_u|\bbf(N_u)=C)& \le \frac{1}{d}\left[(d-d'(u))\log c+d'(u)\log(c')\right]+\log (q-c)\\
&\le \log c+\frac{d'(u)}{d}\log(1-1/c)+\log(q-c)\\
&=\log\qq-\gO(d'(u)/d).\end{split}\enq
Combining \eqref{nonprinc} and \eqref{princ}, we have
\[ \begin{split} \frac{1}{d}H(\bbf_{N_u}|\bbf(N_u))+H(\bbf_u|\bbf(N_u))&=\sum_C\pr(\bbf(N_u)=C)\left[\frac{1}{d}H(\bbf_{N_u}|\bbf(N_u)=C)+H(\bbf_u|\bbf(N_u)=C)  \right]\\
&\le \log \qq-\gO(d'(u)/d). \end{split} \]

\vspace{-.6cm}

\end{proof}

To finish bounding the first term of \eqref{case1.bd}, it remains to estimate the sizes of the sets $[X_1]^- \cup (S_2 \setminus G_1)$ (where we must use a naive bound) and $G_1\setminus [X_1]^-$ (where we may use the stronger bound from the above proposition). And this will follow from the expansion of the graph; note that by Lemma \ref{lem:isoper} (iii) (and the bound on $g$ in \Cref{lem.step2}.), 
\beq{bound on g} |N(X_1)|\ge (1+\gO(1))|[X_1]|\ge (1+\gO(1))|X_1|. \enq

Let $x_1^-:=|[X_1]^-|$. By \eqref{def_tu} and \Cref{prop5.3},
\beq{case 1. first term 2}\begin{split}\sum_{u \in G_1 \setminus [X_1]^-}T(u) &\le \sum_{u \in G_1 \setminus [X_1]^-} \left[ \frac{1}{d}H(\bbf(N_u))+\log \qq-\gO(d'(u)/d)\right]\\
&\le O(g/d) +(g_1-x_1^-)\log \qq -\gO(1/d)\cdot\sum_{u \in G_1 \setminus [X_1]^-}(d-d_{[X_1]}(u))\\
&=O(g/d)+(g_1-x_1^-)\log \qq-\gO(1/d)\cdot d(g_1-|[X_1]|)\\
&\stackrel{\eqref{bound on g}}{=}(g_1-x_1^-)\log\qq-\gO(g).\end{split}\enq
Note by \eqref{sf.ub} that $|G_1 \cup S_2|\leq g_1+(1 + o(1))g_2\leq (1 + o(1))g$.
Combining \eqref{case 1. first term 1} and \eqref{case 1. first term 2}, we have
\[\sum_{u \in G_1\cup S_2} T(u) \le |G_1 \cup S_2|\log \qq + O(g/d)-\gO(g)= |G_1 \cup S_2|\log \qq - \gO(g).\]

\nin (ii) \textit{Second term of \eqref{case1.bd}}: If $u \not \in S_2$ then $u$ is good, so $H(\bbf_u) \le \log (q/2)$. Therefore,
\[\sum_{u \in \cU \setminus (G_1 \cup S_2)} H(\bbf_u) \le (N/2-|G_1 \cup S_2|)\log(q/2).\]

\nin (iii) \textit{Last term of \eqref{case1.bd}}:  Observe that for any $v \in \cD$,
\beq{eq.1}\left(1-\frac{d_{G_1 \cup S_2}(v)}{d}\right)H(\bbf_v) \le \left(1-\frac{d_{G_1 \cup S_2}(v)}{d}\right)\log\left(\frac{q}{2}\right);\enq
indeed, if $v \in [X_1]$, then $d_{G_1}(v)=d$ so the inequality holds (with equality); if $v \not\in [X_1]$, then $v$ is good, so $H(f_v) \le \log(q/2)$. By this observation, the last term of \eqref{case1.bd} is at most
\[\frac{1}{d}\sum_{v \in \cD} (d-d_{G_1 \cup S_2}(v))\log\left(\frac{q}{2}\right) \le \log \left(\frac{q}{2}\right) \cdot \frac{1}{d} \cdot d\left(\frac N2-|G_1 \cup S_2|\right)=\log\left(\frac{q}{2}\right)\left(\frac N2-|G_1 \cup S_2|\right).\]
Combining the bounds in (i)-(iii), \eqref{case1.bd} is at most
\beq{step2.conclusion1}|G_1 \cup S_2|\log\left(\frac{q}{2}\right)^2-\gO(g)+\left(\frac N2-|G_1 \cup S_2|\right)\log \left(\frac{q}{2}\right) +\left(\frac N2-|G_1 \cup S_2|\right)\log \left(\frac{q}{2}\right)=N\log\left(\frac{q}{2}\right)-\gO(g).\enq
Finally, the combination of \eqref{NX.cost} and \eqref{step2.conclusion1} gives that in Case 1, i.e., $g_1 - f_1 < \eps g_1$, the number of colorings in $\cF_{F, S}(g_1)$ is at most \[
2^{N\log(q/2)-\gO(g)}.\]

\noindent \textbf{Case 2. $g_1-f_1 \ge \eps g_1$ and $f_1-s_1 < \eps' g_1$}  

 Recall that we are given $(F,S)$, thus $(F_1,S_1)$ and $(F_2,S_2)$. Note that for any coloring $f\in \cF_{F, S}(g_1)$,
\beq{s gives g} \mbox{if $f_{S_1}$ is specified, then this gives $N(X_1)$,}\enq
because $f_{S_1}$ gives $X_1$ by $X_1=\{v\in S_1: f_v\in A\}$, thus $N(X_1)$. Now let $\bbf$ be uniformly chosen from $\cF_{F, S}(g_1)$ ; our plan for Case 2 is to \textit{first} disclose information on $\bbf_{S_1}$ to specify $N(X_1)$ (not $\bbf_{N(X_1)}$), and then take advantage of the fact that $N(X_1) \setminus F_1$ is somewhat large.

 We first specify $\bbf$ on $F_1 \cup S_1$, whose entropy cost is
\[\begin{split}
H(\bbf_{F_1 \cup S_1})&\stackrel{\eqref{entropyprop2}}{=}H(\bbf_{S_1})+H(\bbf_{F_1}|\bbf_{S_1})\\
&\stackrel{(\dagger)}{\le} \frac1d \sum_{u \in F_1} H(\bbf_{N_u \cap S_1})+\sum_{v \in S_1} \left(1-\frac{d_{F_1}(v)}{d} \right)H(\bbf_v)+\sum_{u \in F_1} H(\bbf_u|\bbf_{N_u \cap S_1})\\
&= \sum_{u \in F_1} \left[ \frac1d H(\bbf_{N_u \cap S_1})+H(\bbf_u|\bbf_{N_u \cap S_1})\right]+\sum_{v \in S_1} \left(1-\frac{d_{F_1}(v)}{d} \right)H(\bbf_v)
\end{split}\]
where for $(\dagger)$ we apply Lemma \ref{lem:Sh} with $\alpha_T=1/d$ if $T=N_u \cap S_1$ for some $u \in F_1$ and $\alpha_v=1- d_{F_1}(v)/d$ for each singleton $v \in S_1$ (so that \eqref{fractiling} is satisfied). Note that, for each $u \in F_1$ and a possible value $C$ for $\bbf(N_u \cap S_1)$, a similar argument as in the proof of Proposition \ref{entropy.bm} gives that
\[\begin{split}
\frac1d H(\bbf_{N_u \cap S_1})+H(\bbf_u|\bbf_{N_u \cap S_1})
&\le \frac1d H(\bbf(N_u \cap S_1))+ \frac1d H(\bbf_{N_u \cap S_1}|\bbf(N_u \cap S_1)) +H(\bbf_u|\bbf(N_u \cap S_1))\\
& \le  O(1/d) + \log(q/2)^2 .
\end{split}\]
Also note by Observation~\ref{prop:SFappro1} and \eqref{approx2} that
\[\sum_{v \in S_1} \left(1-\frac{d_{F_1}(v)}{d} \right)H(\bbf_v) \le \sum_{v \in S_1} \left(\psi/d\right)H(\bbf_v) =O\left(g\psi/d\right),\]
thus
\beq{tired'} H(\bbf_{F_1 \cup S_1}) \le f_1\log(q/2)^2+O(g\psi/d).\enq

In what follows we use the random variable $\tilde \bbf:=\bbf\mid f_{F_1 \cup S_1}$ for the distribution of $\bbf$ conditioned on the event that we have specified colors on $F_1 \cup S_1$. (Eventually we give an upper bound on $\tilde \bbf_{(\cD \cup \cU) \setminus (F_1 \cup S_1)}$ that is valid regardless of the coloring of $F_1 \cup S_1$.) We use $G_1$ for $N(X_1)$ given by the color specification on $S_1$ (see \eqref{s gives g}). Below is our key lemma.

\begin{lem}\label{end}
For any $u \in G_1 \setminus F_1$,
\beq{end.st}\frac{1}{d} H(\tilde \bbf_{N_u \setminus S_1})+H(\tilde \bbf_u|\tilde \bbf_{N_u\setminus S_1}) \le \log \qq-\gO(1).\enq
\end{lem}

\begin{proof}
Note that the lhs of \eqref{end.st} is at most
\beq{end.re}\frac{1}{d}H(\tilde \bbf(N_u \setminus S_1))+\frac{1}{d}H(\tilde \bbf_{N_u \setminus S_1}|\tilde \bbf (N_u \setminus S_1))+H(\tilde \bbf_u|\tilde \bbf(N_u \setminus S_1)),\enq
and if $\tilde \bbf(N_u \setminus S_1)=C$ for a non-principal color set $C$, then (with $c:=|C|$)
\beq{nonprinc'} \frac{1}{d}H(\tilde \bbf_{N_u \setminus S_1}|\tilde \bbf(N_u \setminus S_1)=C)+H(\tilde \bbf_u|\tilde \bbf(N_u \setminus S_1)=C) \le \log c +\log (q-c) =\log \qq-\gO(1).\enq
Now, suppose $\tilde \bbf(N_u \setminus S_1)=C$ for a principal color set $C$. Note that we have already specified colors on $S_1$, and if some vertices in $S_1 \cap N_u$ use colors not in $C$, then we again  have
\beq{princ''} \frac{1}{d}H(\tilde \bbf_{N_u \setminus S_1}|\tilde \bbf(N_u \setminus S_1)=C)+H(\tilde \bbf_u|\tilde \bbf(N_u \setminus S_1)=C) \le \log (q/2)+\log (q/2-1)= \log \qq-\gO(1).\enq

So now assume that $\tilde \bbf(N_u \setminus S_1)=C$ and that $C$ includes all the colors on $N_u \cap S_1$.  Note that $|N_u \setminus S_1| \ge d-\psi$ by \eqref{approx3}, so in particular, $u \in G_1 \setminus [X_1]^-$. (Recall the definition of $[X_1]^-$ from \eqref{xminus}.)
This implies $u \sim X_1$ and $u \sim \cD \setminus X_1$, so $C$ must intersect both of $A$ and $B$. Let $C' = C \cap B ~ (\ne \emptyset)$ and $c'=|C'| (\le c-1)$. Note that the vertices in $N_u \setminus S_1$ are all good, thus each vertex in $N_u \setminus S_1$ has at most $c'$ choices for colors. Therefore, with $d':=|N_u \setminus S_1| (\ge d-\psi)$, we have
\beq{princ'}\begin{split} \frac{1}{d}H(\tilde \bbf_{N_u \setminus S_1}|\tilde \bbf(N_u \setminus S_1)=C)+H(\tilde \bbf_u|\tilde \bbf(N_u\setminus S_1)=C)& \le \frac{d'}{d}\log(c')+\log (q-c)\\
&\le \log c+(d'/d)\log(1-1/c)+\log(q-c)\\
&=\log\qq-\gO(1).\end{split}\enq
Combining \eqref{nonprinc'}, \eqref{princ''}, and \eqref{princ'}, we have that the last two terms of \eqref{end.re} are at most
\[ \begin{split} 
\sum_C\pr(\tilde \bbf(N_u\setminus S_1)=C)\left[\frac{1}{d}H(\tilde \bbf_{N_u\setminus S_1}|\tilde \bbf(N_u \setminus S_1)=C)+H(\tilde \bbf_u|\tilde \bbf(N_u \setminus S_1)=C)  \right]\le \log \qq-\gO(1). \end{split} \]
Therefore, \eqref{end.re} is at most $\log \qq-\gO(1)$ (by noting that its first term is $O(1/d)$).
\end{proof}
~

Now we consider
\beq{tired}\begin{split} 
H(\tilde \bbf_{(\cD \cup \cU) \setminus (F_1 \cup S_1)}) = H(\tilde \bbf_{\cD \setminus S_1})+H(\tilde \bbf_{\cU \setminus F_1}|\tilde \bbf_{\cD \setminus S_1}).
\end{split}\enq
We bound the first term on the rhs by applying Lemma \ref{lem:Sh} with $\alpha_S=1/d$ if $S=N_u \setminus S_1$ for an $u \in (G_1 \cup S_2) \setminus F_1$, and some $\alpha_v$ for $v \in \cD \setminus S_1$ so that \eqref{fractiling} is satisfied. Noting that $H(\tilde \bbf_v) \le \log(q/2)$ for each vertex $ v \in \cD \setminus S_1$ (since they are good),
\beq{slack.bd2}\begin{split}H(\tilde \bbf_{\cD \setminus S_1}) &\le \frac{1}{d}\sum_{u \in (G_1 \cup S_2)\setminus F_1} H(\tilde \bbf_{N_u \setminus S_1})+\sum_{v \in \cD \setminus S_1} \left(1-\frac{d_{(G_1 \cup S_2)\setminus F_1}(v)}{d}\right) \log\left(\frac q2\right). \end{split}\enq
To bound the second sum of \eqref{slack.bd2}, observe that
\[\begin{split} \sum_{v \in \cD \setminus S_1} d_{(G_1 \cup S_2) \setminus F_1}(v)&=|\nabla(\cD \setminus S_1, (G_1 \cup S_2) \setminus F_1)|\\
& \ge |\nabla(\cD,G_1 \cup S_2)|-|\nabla(\cD, F_1)|-|\nabla(S_1,\cU \setminus F_1)|\\
& \stackrel{\eqref{approx2}}{\ge} d|G_1 \cup S_2|-df_1-\psi s_1 \stackrel{\eqref{sf.ub}}{=} d|G_1 \cup S_2|-df_1-O(\psi g),
\end{split}\]
so \eqref{slack.bd2} is at most (recalling that $|\cD|=N/2$)
\beq{slack.bd2'}\frac{1}{d}\sum_{u \in (G_1 \cup S_2)\setminus F_1} H(\tilde \bbf_{N_u \setminus S_1})+\log\left(\frac q2\right)\left(\frac{N}{2}-|G_1 \cup S_2| + f_1 - s_1\right) + O\left(\frac{\psi g}{d}\right).\enq

\nin For the second term in the rhs of \eqref{tired}, again noting that $H(\tilde \bbf_u) \le \log (q/2)$ for each vertex $u \in \cU \setminus S_2$ (since they are good),
\beq{slack.bd3}\begin{split}H(\tilde \bbf_{\cU \setminus F_1}|\tilde \bbf_{\cD \setminus S_1}) &\le \sum_{u \in (G_1 \cup S_2)\setminus F_1} H(\tilde \bbf_u|\tilde \bbf_{N_u \setminus S_1}) +|\cU \setminus (G_1 \cup S_2)|\log\left(\frac q2\right)\\
&= \sum_{u \in (G_1 \cup S_2)\setminus F_1} H(\tilde \bbf_u|\tilde \bbf_{N_u \setminus S_1}) +\log\left(\frac q2\right)\left(\frac{N}{2}-|G_1 \cup S_2|\right).\end{split}\enq
~
The sum of the first terms of \eqref{slack.bd2'} and \eqref{slack.bd3} are at most, by Proposition \ref{entropy.bm} and Lemma \ref{end},
\beq{too.tired} \begin{split}&\sum_{u \in (G_1 \cup S_2) \setminus F_1} \left[\frac{1}{d} H(\tilde \bbf_{N_u \setminus S_1})+H(\tilde \bbf_u|\tilde \bbf_{N_u \setminus S_1})\right]\\
& \le (g_1-f_1)\left(\log\left(\frac q2\right)^2-\gO(1)\right)+|S_2 \setminus G_1|\left(\log\left(\frac q2\right)^2+O\left(\frac 1d\right)\right)\\
& \le \log\left(\frac q2\right)^2(|G_1 \cup S_2|  - f_1)-\bO(g_1 - f_1) + O\left(\frac gd\right)\\
& \le \log\left(\frac q2\right)^2(|G_1 \cup S_2|  - f_1)-\gO(\eps g),
\end{split}\enq
where for the last inequality we use $\eps:=1/\log^{1.5} d$ and the fact that $g_1-f_1 \ge \eps g_1$ in Case 2.
Combining \eqref{slack.bd2'}, \eqref{slack.bd3}, and \eqref{too.tired} (and observing that $\psi g/d \ll \eps g$), we can bound \eqref{tired} by
\beq{summary} \begin{split} 
\log\left(\frac q2\right)^2\left(\frac{N}{2}-f_1\right) +\log\left(\frac q2\right)(f_1-s_1) -\gO(\eps g)=\log\left(\frac q2\right)^2\left(\frac{N}{2}-f_1\right) -\gO(\eps g)
\end{split}\enq
where the equality uses another assumption in Case 2, i.e., $f_1-s_1 < \eps' g_1$, and $\eps' \ll \eps$.

Finally, combining \eqref{tired'} and \eqref{summary}, we have
\beq{step2.conclusion2}\begin{split}H(\bbf)&=H(\bbf_{F_1 \cup S_1})+H(\bbf_{(\cD \cup \cU) \setminus (F_1 \cup S_1)}|\bbf_{F_1 \cup S_1})\\
& \le H(\bbf_{F_1\cup S_1})+\sum_{\rho}\pr(\bbf_{F_1\cup S_1}=\rho)H(\bbf_{(\cD \cup \cU) \setminus (F_1 \cup S_1)}|\bbf_{F_1 \cup S_1}=\rho)\\
& \le f_1\log(q/2)^2 + O\left(g\psi/d\right)+\log(q/2)^2\left(N/2-f_1\right)-\gO(\eps g) \\
&=N\log (q/2)-\gO(g/\log^{1.5} d),\end{split}\enq
Therefore, in Case 2, the number of colorings in $\cF_{F, S}(g_1)$ is at most \[
2^{N\log (q/2)-\gO(g/\log^{1.5} d)}.\]

\noindent\textbf{Case 3. $f_1-s_1 \ge \eps' g_1$.}

Let $\bbf$ be uniformly chosen from $\cF_{F, S}(g_1)$.
Recall that $(F_1\cup S_2)^+=(F_1\cup S_2) \cup N(F_1 \cup S_2)$.
 Then with $T(u)$ as in \eqref{eq.Tu.def}, we have
\begin{eqnarray}
H(\bbf_{(F_1\cup S_2)^+})&=&H(\bbf_{N(F_1 \cup S_2)})+H(\bbf_{F_1 \cup S_2}\mid\bbf_{N(F_1 \cup S_2)})\nonumber\\
&\stackrel{(\dagger)}{\le}& \sum_{u \in F_1 \cup S_2} \frac{1}{d}H(\bbf_{N_u})+\sum_{v \in N(F_1 \cup S_2)}\left(1-\frac{d_{F_1 \cup S_2}(v)}{d}\right)H(\bbf_v) + \sum_{u \in F_1 \cup S_2}H(\bbf_{u}\mid \bbf_{N_u})\nonumber\\
&=&\sum_{u \in F_1 \cup S_2} T(u)+\sum_{v \in N(F_1 \cup S_2)}\left(1-\frac{d_{F_1 \cup S_2}(v)}{d}\right)H(\bbf_v)\label{step3termM}
\end{eqnarray}
where for $(\dagger)$ we apply Lemma \ref{lem:Sh} with $\alpha_S=1/d$ if $S=N_u$ for some $u \in F_1\cup S_2$ and $\alpha_v=1- d_{F_1\cup S_2}(v)/d$ for each singleton $v \in N(F_1 \cup S_2)$ (so that \eqref{fractiling} is satisfied). 

Let $F_0=\{u \in F_1: N(u) \sub S_1\}$. The proof of below proposition is identical to that of Proposition \ref{prop5.3}, thus we omit.
\begin{prop}\label{Prop5.4} For any $u \in F_1 \setminus F_0$ with $d'(u):=d-d_{S_1}(u)$,
\[\frac1d H(\bbf_{N_u}|\bbf(N_u))+H(\bbf_u|\bbf({N_u}))\le\log(q/2)^2-\gO(d'(u)/d).\]
\end{prop}

Note that
\beq{prop5.4.note}\sum_{u \in F_1 \setminus F_0} d'(u) =|\nabla(F_1,\cD \setminus S_1)|\ge df_1-ds_1\ge d\eps' g_1.\enq
Using Proposition \ref{Prop5.4} and \eqref{prop5.4.note}, we obtain
\[
\begin{split}
\sum_{u \in F_1 \setminus F_0}T(u)&\stackrel{\eqref{def_tu}}{\le}\sum_{u \in F_1 \setminus F_0}\left[\frac{1}{d}H(\bbf(N_u))+\frac{1}{d}H(\bbf_{N_u}|\bbf(N_u))+H(\bbf_u|\bbf(N_u))\right]\\
&\leq O(g/d) + (f_1 - f_0)\log(q/2)^2-\gO(\eps' g).
\end{split}
\]
On the other hand, using the naive bound in Proposition \ref{entropy.bm}, we have
\[
\sum_{u \in F_0 \cup (S_2 \setminus F_1)} T(u)\le O(g/d) +(f_0+|S_2 \setminus F_1|)\log(q/2)^2.
\]
Combining the above two bounds, the first term of \eqref{step3termM} is 
\beq{step3.term1}
\sum_{u \in F_1\cup S_2} T(u) \leq |F_1\cup S_2|\log(q/2)^2-\gO(\eps' g).
\enq

For the second term of \eqref{step3termM}, using the fact that $H(\bbf_v) \le \log(q/2)$ for all $v \in \cD \setminus S_1$ (since they are all good), we have
\beq{step3.term2}
\begin{split}
\sum_{v \in N(F_1 \cup S_2)}\left(1-\frac{d_{F_1 \cup S_2}(v)}{d}\right)H(\bbf_v)
&\le \sum_{v\in S_1}\left(1-\frac{d_{F_1 \cup S_2}(v)}{d}\right)H(\bbf_v) + \sum_{v \in N(F_1 \cup S_2) \setminus S_1}\left(1-\frac{d_{F_1 \cup S_2}(v)}{d}\right)\log\left(\frac{q}{2}\right)\\
&\stackrel{\eqref{approx2}}{\le} O(\psi g/d)+ \sum_{v \in N(F_1 \cup S_2) \setminus S_1}\left(1-\frac{d_{F_1 \cup S_2}(v)}{d}\right)\log\left(\frac{q}{2}\right).
\end{split}
\enq
Observe that
\[
\sum_{v \in N(F_1 \cup S_2) \setminus S_1}d_{F_1 \cup S_2}(v)
=|\nabla(F_1 \cup S_2,N(F_1 \cup S_2) \setminus S_1)|
\ge d|F_1 \cup S_2|-ds_1.
\]
Then \eqref{step3.term2} is at most
\beq{step3.term20}
(|N(F_1 \cup S_2)|-|F_1 \cup S_2|)\log(q/2)+O(\psi g/d).
\enq
Combining \eqref{step3termM}, \eqref{step3.term1}, and \eqref{step3.term20} (and observing that $\psi g/d \ll \eps' g$), we have
\[H(\bbf_{(F_1\cup S_2)^+})\le |(F_1\cup S_2)^+|\log(q/2)-\gO(\eps' g).\]

Finally, using the fact that $H(\bbf_v)\le \log(q/2)$ for all $v \notin S_1 \cup S_2$ (since such $v$'s are good),
\[
H(\bbf)\le H(\bbf_{(F_1\cup S_2)^+})+\sum_{v \notin (F_1\cup S_2)^+} H(\bbf_v)
\le N\log (q/2)-\gO(g/\log^2 d). \]
Therefore, in Case 3, the number of colorings in $\cF_{F, S}(g_1)$ is at most \[
2^{N\log (q/2)-\gO(g/\log^{2} d)}.\]

\nin \textit{Conclusion.} 
To sum up, we have shown that for any $(F, S)$ and $g_1$, the number of colorings in $\cF_{F, S}(g_1)$ is at most 
\[2^{N\log (q/2)-\gO(g/\log^{2} d)}=(q/2)^N 2^{-\gO(g/\log^2d)}.\]
Note by Lemma \ref{lem:container} and \eqref{psi.bd} that
\[
|\cU(g)|\leq 2^{O(g\log^2 d/d+g\log d/\psi +d)} = 2^{o(g/\log^2 d)}.
\]
Therefore, for any principal partition $(A, B)$ and $g \in [d^{10}, 2^{-\beta' d}N]$, we have
\[
\begin{split}
\sum_{X\in\cH(g)}\chi_{A, B}(X)
&\stackrel{\eqref{qtyFinalBdown}}{=}\sum_{(F, S)\in \cU(g)}\sum_{g_1\in[g/2, g]}2\cdot |\cF_{F, S}(g_1)|
\le 2\cdot 2^{o(g/\log^2 d)} \cdot g\cdot(q/2)^N2^{-\gO(g/\log^2d)}\\
&\le (q/2)^N 2^{-\gO(g/\log^2 d)},
\end{split}
\]
which completes the proof of Lemma~\ref{lem.step2}.

\section{Polymer models on $\cB_d$}\label{sec.polymer models}

We first define a polymer model on $\cB_d$ and introduce relevant definitions. Recall the definitions in \Cref{subsec.polymer intro}. The set of polymers is defined to be
\begin{equation}\label{def:polymer}
\cP:=\left\{\gamma\subseteq V: \gamma \text{ is non-empty, 2-linked, and } |N(\gamma)|\leq 2^{-\beta d}N\right\}
\end{equation}
where $\beta$ is as in \Cref{lem.step1}. Define $H_{\cP}$ to be the graph on $\cP$ where two polymers $\gamma$, $\gamma'$ are  adjacent iff $\gamma\cup \gamma'$ is a $2$-linked set (so every vertex has a loop). Recall that $\Omega_{\cP} (\ni \emptyset)$ is the collection of independent sets of $H_{\cP}$ with loops allowed.

Fix a principal partition $(A, B)$ of $\cQ$. 
In this section we tentatively drop the assumption that $q$ is even, as all proofs here works simultaneously for both even and odd $q$.
We equip each polymer in $\cP$ with the weight function (recall $\chi_{A,B}$ from \eqref{def.chi})
\begin{equation}\label{def:wefun}
\we_{A,B}(\gamma):=\frac{|\chi_{A, B}(\gamma)|}{(|A|\cdot|B|)^{N}} = \frac{|\chi_{A, B}(\gamma)|}{(\lfloor q/2 \rfloor\cdot\lceil q/2 \rceil)^{N}}.
\end{equation}
For future reference, observe that, writing $f_{\gamma^+}$ for $f$ restricted on $\gamma^+$ and $\hat{\chi}_{A, B}(\gamma):=\{ f_{\gamma^+} : f\in \chi_{A, B}(\gamma)\}$,
we can rewrite \eqref{def:wefun} as
\beq{eq.wefun.reform}
\omega_{A,B}(\gamma)=\frac{|\hat{\chi}_{A, B}(\gamma)|}{|A|^{|\gamma^+\cap \mathcal{L}_d|}\cdot|B|^{|\gamma^+\cap\mathcal{L}_{d-1}|}}.
\enq

Now define the polymer model associate to a principal partition $(A, B)$ as
\[
\Xi_{A,B}:=\Xi(\cP, \we_{A, B})=\sum_{\Lambda\in\Omega_{\cP}}\prod_{\gamma\in\Lambda}\we_{A, B}(\gamma).
\]
For a cluster $\Gamma=(\gamma_1, \gamma_2, \ldots, \gamma_k)$ where $\gamma_i\in \cP$, define the \textit{size} of $\Gamma$ to be $\lVert \Gamma\rVert=\sum_{i=1}^k|\gamma_i|$.
For any $k\geq 1$, the \textit{$k$-th term of the cluster expansion} of the polymer model $(\cP, \we_{A, B})$ is
\begin{equation}\label{def:LAB}
L_{A, B}(k):=\sum_{\Gamma\in \mathcal{C},\  \lVert \Gamma\rVert=k}\we_{A,B}(\Gamma)
\end{equation}
(where $\we_{A,B}(\Gamma)$ is defined as in \eqref{def:cluweight})
and then by \eqref{eq:cluexp},
\begin{equation}\label{def:Lsum}
 \ln\Xi_{A,B} = \sum_{\Gamma\in\mathcal{C}}\we_{A,B}(\Gamma) = \sum_{k=1}^{\infty}L_{A, B}(k). 
\end{equation}
\subsection{Verifying Koteck\'{y}--Preiss Condition}

For every polymer $\gamma\in\cP$, consider a new weight function
\[
\tilde{\we}_{A,B}(\gamma):=\we_{A,B}(\gamma)\exp(|\gamma|/d)~(>\we_{A,B}(\gamma))
\]
and the corresponding polymer model partition function
\begin{equation}\label{def:genepmodel}
\tilde{\Xi}_{A,B}:=\Xi(\cP, \tilde{\we}_{A, B})=\sum_{\Lambda\in\Omega_{\cP}}\prod_{\gamma\in\Lambda}\tilde{\we}_{A, B}(\gamma).
\end{equation}
In the next lemma we verify Koteck\'{y}--Preiss Condition for $(\cP, \tilde \we_{A, B})$ (thus for $(\cP, {\we}_{A,B})$). This stronger result will be used in the proof of Lemma~\ref{lem:defectsize}.
\begin{lem}\label{lem:KP}
There exists functions $f: \cP \rightarrow [0, \infty)$ and $g: \cP \rightarrow [0, \infty)$ such that for each principal partition $(A, B)$ and all polymers $\gamma\in\cP$,
\begin{equation}\label{eq:kot-thm}
\sum_{\gamma'\sim \gamma}\we_{A, B}(\gamma')\exp(|\gamma'|/d)\exp\left(f(\gamma')+g(\gamma')\right)\leq f(\gamma).
\end{equation}
In particular, the cluster expansion of both $\ln\Xi_{A, B}$ and $\ln\tilde{\Xi}_{A, B}$ converge absolutely.
\end{lem}
\begin{proof}
Fix a principal partition $(A, B)$, and simply denote $\we_{A, B}$ as $\we$. Recall from \Cref{lem.step2} that there exists a constant $\xi>0$ such that for all $g \in [d^{10}, 2^{-\beta d}N]$,
\begin{equation}\label{ineq:polyertotalweight}
\sum_{\gamma \in \cH(g)} \we(\gamma) \leq \exp(-\xi g/\log^2 d),
\end{equation}
where $\cH(g)=\{X \sub V: X \mbox{ is 2-linked}, |N(X)|=g\}$.
Let
\begin{equation}\label{def:fgfun}
f(\gamma)=|\gamma|/d \quad \text{and} \quad g(\gamma)=\begin{cases}
|N(\gamma)|/3q  & \text{if } |N(\gamma)|\leq d^{10} \\ 
\xi |N(\gamma)|/(4\log ^2d) & \text{otherwise}.
\end{cases}
\end{equation}
Note that, to show \eqref{eq:kot-thm}, it suffices to show that for each $v\in V$,
\begin{equation}\label{ineq:convercond}
\sum_{\gamma':\ \gamma'\ni v}\we(\gamma')\cdot \exp(2f(\gamma')+g(\gamma'))\leq 1/d^3, 
\end{equation}
since the lhs of \eqref{eq:kot-thm} is at most
\[
\sum_{u\in \gamma}\sum_{v\in N^2(u)}\sum_{\gamma':\ \gamma'\ni v}\we(\gamma')\cdot \exp(2f(\gamma')+g(\gamma'))\stackrel{\eqref{ineq:convercond}}{\le} |\gamma|d^2\cdot (1/d^3)=f(\gamma).
\]

\nin We split the sum in \eqref{ineq:convercond} into two parts according to $|N(\gamma')|$. 
First, for $\gamma$ with $|N(\gamma)| >d^{10}$, 
\beq{eq.case1}
\begin{split}
\sum_{\gamma':\ \gamma'\ni v,\ |N(\gamma')|\geq d^{10}}\we(\gamma')\cdot \exp(2f(\gamma')+g(\gamma'))
&\leq \sum_{g=d^{10}}^{2^{-\beta d}N}\sum_{\gamma':\ \gamma'\in \cH(g)}\we(\gamma')\cdot \exp(2f(\gamma')+g(\gamma'))\\
&\stackrel{\eqref{ineq:polyertotalweight}}{\leq} \sum_{g=d^{10}}^{2^{-\beta d}N}\exp(-\xi g/\log^2 d)\cdot\exp(\xi g/(2\log^2 d)) \leq \frac{1}{2d^3}
\end{split}
\enq
for $d$ sufficiently large. Next, for $\gamma$ with $|N(\gamma)| \le d^{10}$, we first give an upper bound on $\we(\gamma)$.  
Observe that a naive counting gives
\[
\begin{split}
|\hat{\chi}_{A, B}(\gamma)| 
&\leq |A^c|^{|\gamma\cap \cL_{d}|}(|B|-1)^{|\partial(\gamma)\cap \cL_{d-1}|}|B^c|^{|\gamma\cap \cL_{d-1}|}(|A|-1)^{|\partial(\gamma)\cap \cL_{d}|}\\
&= |B|^{|\gamma\cap \cL_{d}|}(|B|-1)^{|\partial(\gamma)\cap \cL_{d-1}|}|A|^{|\gamma\cap \cL_{d-1}|}(|A|-1)^{|\partial(\gamma)\cap \cL_{d}|},
\end{split}
\]
therefore,
\begin{equation}\label{ineq:polyweight}
\begin{split}
\we(\gamma) &\stackrel{\eqref{eq.wefun.reform}}{\leq} \left(\frac{|B|}{|A|}\right)^{|\gamma\cap \cL_{d}|}\left(\frac{|B|-1}{|B|}\right)^{|\partial(\gamma)\cap \cL_{d-1}|}\left(\frac{|A|}{|B|}\right)^{|\gamma\cap \cL_{d-1}|}\left(\frac{|A|-1}{|A|}\right)^{|\partial(\gamma)\cap \cL_{d}|}\\
&\leq \left(\max\left\{\frac{|B|}{|A|},\ \frac{|A|}{|B|}\right\}\right)^{|\gamma|}\left(\max\left\{1 - \frac{1}{|A|},\ 1 - \frac{1}{|B|}\right\}\right)^{|\partial(\gamma)|}\\
&\leq \left(1 + \frac{1}{\lfloor q/2 \rfloor}\right)^{|\gamma|}\left(1 - \frac{1}{\lceil q/2 \rceil}\right)^{|\partial(\gamma)|}.
\end{split}
\end{equation}
Now, if $|N(\gamma)|=g\leq d^{10}$, then by \Cref{lem:isoper} (ii), we have 
$|\gamma|\leq 12g/d$. So we have
\[
\begin{split}
\sum_{\substack{\gamma':\ \gamma'\ni v \\   
\   |N(\gamma')|\leq d^{10}}}\we(\gamma')\exp(2f(\gamma')+g(\gamma'))
&\leq \sum_{g=d}^{d^{10}}\sum_{\gamma':\ \gamma'\ni v,\ |N(\gamma')|=g}\we(\gamma')\exp(2f(\gamma')+g(\gamma'))\\
&\stackrel{\eqref{eq:numlinkset}, \eqref{ineq:polyweight}}{\leq} \sum_{g=d}^{d^{10}}(ed^2)^{12g/d}\left(1 + \frac{1}{\lfloor q/2 \rfloor}\right)^{12g/d}\left(1 - \frac{1}{\lceil q/2 \rceil}\right)^{g/2}\exp\left(\frac{24g}{d^2} + \frac{g}{3q}\right)\\
&\leq \sum_{g=d}^{d^{10}}(2e^2d^2)^{12g/d}\left(1 - \frac{1}{\lceil q/2 \rceil}\right)^{g/2}\exp\left(\frac{g}{3q}\right) \le \frac{1}{2d^3}.
\end{split}
\]
for $d$ large enough. This, together with \eqref{eq.case1}, completes the proof.
\end{proof}

\subsection{Bounding terms in cluster expansion} For a principal partition $(A, B)$, define
\[
\delta_{A, B}:=\max\left\{1 - \frac{1}{|A|}, 1 - \frac{1}{|B|}\right\}=1 - \frac{1}{\lceil q/2 \rceil}.
\]

\begin{lem}\label{lem:erroresti}
Let $(A, B)$ be a principal partition. Then for any fixed $t\geq 1$, we have 
\[
\sum_{k=t}^{\infty}|L_{A, B}(k)|=\bO_t\left(Nd^{2(t-1)}\delta_{A, B}^{dt}\right).
\]
\end{lem}
\begin{proof}
We show that
\beq{claim1 of lem 5.2}
|L_{A, B}(k)|=\bO_k\left(Nd^{2(k-1)}\delta_{A, B}^{dk}\right) \mbox{ for each $k\in\NN$;}
\enq
and
\beq{claim2 of lem 5.2}
\mbox{for each $t\geq 1$, there exists a constant $K=K(t)$ such that 
$\sum_{k>K}|L_{A, B}(k)|\le \delta_{A, B}^{dt}$,}
\enq
from which the conclusion follows since
\[
\sum_{k=t}^{\infty}|L_{A, B}(k)|=\sum_{k=t}^{K}|L_{A, B}(k)| +\sum_{k>K}|L_{A, B}(k)|=
\bO_t\left(Nd^{2(t-1)}\delta_{A, B}^{dt}\right).
\]

\nin For \eqref{claim1 of lem 5.2}, let $\gamma$ be a polymer of size at most a fixed $k\in\NN$. By \Cref{lem:isoper} (i), we have $|N(\gamma)|\geq d|\gamma| - \bO(1)$, so
\[
\begin{split}
\we_{A,B}(\gamma)&\stackrel{\eqref{ineq:polyweight}}{\leq} \left(1 + \frac{1}{\lfloor q/2 \rfloor}\right)^{|\gamma|}\left(1 - \frac{1}{\lceil q/2 \rceil}\right)^{|\partial(\gamma)|}\\
&\leq\left(1 + \frac{1}{\lfloor q/2 \rfloor}\right)^{k}\left(1 - \frac{1}{\lceil q/2 \rceil}\right)^{d|\gamma|- |\gamma|-\bO(1)}=\bO_k\left(\delta_{A, B}^{d|\gamma|}\right).\\
\end{split}
\]
Therefore, for a cluster $\Gamma$ of size $\lVert \Gamma\rVert=k$, we have
\begin{equation}\label{eq:clutribound}
\we_{A,B}(\Gamma)=\phi(H_{\cP}(\Gamma))\prod_{\gamma\in \Gamma}\we_{A,B}(\gamma)\stackrel{\eqref{ursell}}{=}\bO_k\left(\delta_{A, B}^{dk}\right).
\end{equation}
Next, observe that for a cluster $\Gamma$ with $\lVert \Gamma\rVert=k$, $V(\Gamma)$ is a 2-linked set of size at most $k$ (by the definition of the cluster). Therefore, by \Cref{lem:numlinkset}, the number of options for $V(\Gamma)$ is at most $\bO_k\left(N\cdot d^{2(k-1)}\right)$.
Notice that given $V(\Gamma)$, there are only $\bO(1)$ possibilities for $\Gamma$.
By putting all ingredients together, we have that, for a fixed integer $k$, 
\begin{equation}\label{ineq:fixk}
L_{A, B}(k)=\sum_{\Gamma\in \mathcal{C},\  \lVert \Gamma\rVert=k}\we_{A,B}(\Gamma) = \bO_k\left(N d^{2(k-1)}\delta_{A, B}^{dk}\right).
\end{equation}

Next, we show \eqref{claim2 of lem 5.2}. 
By \Cref{lem:KP} and ~\eqref{eq:kot-thm2}, for any vertex $v\in V$ and $g(\gamma)$ as in \eqref{def:fgfun},
\[
\sum_{\Gamma\in \mathcal{C}, \Gamma\sim \{v\}}|\we_{A,B}(\Gamma)|\exp\left(g(\Gamma)\right)\leq 1/d.
\]
Trivially every cluster $\Gamma$ is adjacent to some vertex in $V$, so
\beq{N over d}
\sum_{\Gamma\in \mathcal{C}}|\we_{A,B}(\Gamma)|\exp\left(g(\Gamma)\right)\leq N/d.
\enq
By \Cref{lem:isoper} (ii) and the definition of $g(\gamma)$, we have
\[
g(\gamma)\geq
\begin{cases}
d|\gamma|/(36q) & \text{ if } |\gamma|\leq d^9\\
d^8 & \text{ otherwise}.
\end{cases}
\]
Therefore, for each $k\in\NN$ and any cluster $\Gamma$ of size at least $k$, we have $g(\Gamma)=\sum_{\gamma\in\Gamma}g(\gamma)\geq dk/(36q)$. 
Now, take $K=K(t)$ to be the smallest integer that satisfies
\[
\exp\left(-d(K+1)/(36q))\right) \leq \delta_{A, B}^{dt}\cdot \frac{d}{N}
\]
(such $K$ exists since $N\sim 2^{2d}/\sqrt{\pi d}$). Then we have
\[
\begin{split}
\sum_{k>K}|L_{A, B}(k)|&\leq \sum_{\Gamma\in \mathcal{C},\  \lVert \Gamma\rVert\geq K+1}|\we_{A,B}(\Gamma)|=\sum_{\Gamma\in \mathcal{C},\  \lVert \Gamma\rVert\geq K+1}|\we_{A,B}(\Gamma)|\exp\left(g(\Gamma)\right)\exp\left(-g(\Gamma)\right)\\
&\leq \delta_{A, B}^{dt}\cdot \frac{d}{N}\cdot\sum_{\Gamma\in \mathcal{C}}|\we_{A,B}(\Gamma)|\exp\left(g(\Gamma)\right) \stackrel{\eqref{N over d}}{\leq} \delta_{A, B}^{dt}\cdot \frac{d}{N}\cdot \frac{N}{d}=\delta_{A, B}^{dt}.
\end{split}
\]
\end{proof}

\subsection{Almost all colorings are captured by at most one polymer model}
For a principal parition $(A, B)$, we say that a coloring $f\in C_q(\cB_d)$ is \textit{captured} by the polymer model $\Xi_{A, B}$ if every 2-linked components of $X_{A, B}(f)$ is of size at most $2^{-\beta d}N$ (where $\beta$ is as in \Cref{lem.step1}). By the definition of $\Xi_{A, B}$, 
\begin{equation}\label{ineq:capture}
\text{the number of colorings captured by $\Xi_{A, B}$ is exactly }
(\lfloor q/2 \rfloor\lceil q/2 \rceil)^{N/2}\cdot\Xi_{A, B}.
\end{equation}
The following lemma says that almost all colorings are captured by at most one polymer model.
\begin{lem}\label{lem:capture1}
Let $q \geq 4$, and denote by $C_2$ the set of colorings in $C_q(\cB_d)$ captured by at least two of the polymer models $\Xi_{A, B}$'s. Then
\beq{eq:lem:capture1}
|C_2| \leq \exp\left(-\Omega\left(N/d^2\right)\right) (\lfloor q/2 \rfloor\lceil q/2 \rceil)^{N/2} \sum_{(A, B) \text{ principal}}\Xi_{A, B}.
\enq
\end{lem}
\nin We will use \Cref{lem:defect.shorten} below to prove \Cref{lem:capture1}. The proof of \Cref{lem:defect.shorten} is almost identical to the proof of \cite[Lemma~11.2]{JK} and we put it in \Cref{appendix} for the sake of completeness.

As in \cite{JK}, it is helpful to consider an auxiliary probability measure $\hat{\mu}_q$ on $C_q(\cB_d)$ (instead of uniform probability measure).

\begin{mydef} \label{def:muh}
For an $f \in C_q(\cB_d)$, let $\hat{\mu}_q(f)$ denote the probability that $f$ is selected by the following four-step process:

\begin{itemize}
    \item[1.] Choose a principal partition $(A, B)$ uniformly at random;
    \item[2.] Choose $\Lambda \in \Omega_{\cP}$ from the distribution $\nu_{A, B}$, where
\begin{equation}\label{def:rd}
\nu_{A, B}(\Lambda):=\frac{\prod_{\gamma\in\Lambda}\we_{A, B}(\gamma)}{\Xi_{A, B}} \quad \mbox{for each $\Lambda\in\Omega_{\cP}$};
\end{equation}
    \item[3.] With $S:=\bigcup_{\gamma\in\Lambda}\gamma$, select a coloring $f\in \hat{\chi}_{A, B}(S)$ uniformly at random;
    \item[4.] Independently assign each $v\in \cL_d\setminus S^+$ ($v\in \cL_{d-1}\setminus S^+$, resp.) the color $c\in A$ with probability $1/|A|$ (the color $c\in B$ with probability $1/|B|$, resp.).
\end{itemize}
\end{mydef}

For $s>0$, say $f$ is \textit{$s$-balanced with respect to $(A, B)$} if for each $c\in A$ ($c \in B$, resp.), the proportion of vertices of $\cL_d$ ($\cL_{d-1}$, resp.) colored $c$ is $1/|A|\pm s$ ($1/|B|\pm s$, resp.). The following lemma says that a typical coloring sampled from the measure $\hat{\mu}_q$ is well-balanced with respect to some principal partition.

\begin{lem}\label{lem:defect.shorten}
Let $\mathbf f$ be a random coloring drawn from the distribution $\hat{\mu}_q$, and denote by $\mathbf D$ the principal partition selected at Step 1.
Then there is a constant $L=L(q)$ for which the following holds: if $(A, B)$ is a principal partition and $Ld^2\delta_{A, B}^d \leq s\leq 1$, then
\[
\pr\left(\text{$\mathbf f$ is $s$-balanced w.r.t. $(A, B)$} \mid \mathbf D=(A, B) \right) \geq 1 - \exp\left(-\Omega\left(s^2N\right)\right) - \exp\left(-\Omega\left(sN/d^2\right)\right).
\]

\end{lem}

\begin{proof}[Proof of Lemma~\ref{lem:capture1}]
Let $\mathbf f$ and $\mathbf D$ be as in \Cref{lem:defect.shorten}. 
Fix two distinct principal color partitions $(A, B)$ and $(A', B')$, and let $\cF$ be the set of colorings which are captured by both $\Xi_{A, B}$ and $\Xi_{A', B'}$. Set $s=1/q^2$ and observe that we may partition $\cF=\cF_{A, B}\cup \cF_{A', B'}$, where $\cF_{A, B}$ ($\cF_{A', B'}$, resp.) consist of elements of $\cF$ that are \textit{not} $s$-balanced w.r.t. $(A, B)$ ($(A', B')$, resp.) (because no coloring can be $s$-balanced w.r.t. two distinct principal partitions).

It follows from \Cref{lem:defect.shorten} that
\begin{equation}\label{ineq:capture1}
\pr\left( \mathbf f\in\cF_{A, B} \mid \mathbf D=(A, B) \right) \leq \exp\left(-\Omega\left(N/d^2\right)\right).
\end{equation}
On the other hand, for any $f$ that is captured by $\Xi_{A, B}$, let the $2$-linked components of $X_{A, B}(f)$ be $\gamma_1, \ldots, \gamma_k$, and write $\Lambda=\{\gamma_1, \ldots, \gamma_k\}$. Then by the definition of $\hat{\mu}_q$, we have
\[
\pr\left( \mathbf f=f \mid \mathbf D=(A, B) \right) = \frac{\prod_{\gamma\in \Lambda}\we_{A, B}(\gamma)}{\Xi_{A,B}}\cdot\frac{1}{|\chi_{A, B}(X_{A, B}(f))|} = \frac{1}{\Xi_{A,B}\cdot(\lfloor q/2 \rfloor\lceil q/2 \rceil)^{N/2}}.
\]
and therefore
\[
\pr\left( \mathbf f\in\cF_{A, B} \mid \mathbf D=(A, B) \right)=\sum_{f\in \cF_{A, B}}\pr\left( \mathbf f=f \mid \mathbf D=(A, B) \right) =\frac{|\cF_{A, B}|}{\Xi_{A,B}\cdot(\lfloor q/2 \rfloor\lceil q/2 \rceil)^{N/2}}.
\]
This, together with~\eqref{ineq:capture1}, shows that
\[
|\cF_{A, B}| \leq \exp\left(-\Omega\left(N/d^2\right)\right)(\lfloor q/2 \rfloor\lceil q/2 \rceil)^{N/2}\Xi_{A,B}.
\]
Similar bound can be obtained for $|\cF_{A', B'}|$, so 
\[
|C_2| \leq \exp\left(-\Omega\left(N/d^2\right)\right)(\lfloor q/2 \rfloor\lceil q/2 \rceil)^{N/2}\sum_{(A, B) \text{ principal}}\Xi_{A, B}.
\]
\end{proof}

\section{Proof of \Cref{realMT}}\label{sec.realMT}
Recall from \eqref{def:threpolysize} that $T(q)$ is the smallest integer $t$ such that $2 + t\log(1 - 2/q)< 0$.
Combining Lemmas \ref{lem.step1}, \ref{lem:erroresti}, and \ref{lem:capture1}, we obtain the following theorem on $c_q(\cB_d)$, the number of proper colorings of $\cB_d$.
\begin{thm}\label{MT1'}
Let $q \ge 4$ and $q$ be even. Then we have 
\[
c_q(\cB_d) = (1 + o(1))(q/2)^{N} \sum_{(A, B) \text{ principal}}\exp\left(\sum_{k=1}^{t-1}L_{A, B}(k) + \varepsilon_t\right),
\]
where $L_{A, B}(k)$ is as defined in ~\eqref{def:LAB} and  
\beq{eps.bound} \varepsilon_t = \bO\left(Nd^{2(t-1)}\delta_{A, B}^{dt}\right).\enq
In particular, 
\[
c_q(\cB_d) = (1 + o(1))(q/2)^{N} \sum_{(A, B) \text{ principal}}\exp\left(\sum_{k=1}^{\thre(q)-1}L_{A, B}(k)\right).
\]
\end{thm}
\begin{proof}
By \Cref{lem.step1}, all but $2^{-\alpha d}c_q(\cB_d)$ colorings are captured by at least one of the polymer models $\Xi_{A, B}$'s.
Therefore,
\[
(1-o(1))(q/2)^{N}\sum_{(A, B) \text{ principal}}\Xi_{A, B} \stackrel{\eqref{eq:lem:capture1}}{\leq} c_q(\cB_d) \stackrel{\eqref{ineq:capture}}{\leq} (q/2)^{N}\sum_{(A, B) \text{ principal}}\Xi_{A, B} + 2^{-\alpha d}c_q(\cB_d),
\]
thus
\[
c_q(\cB_d) = (1 + o(1))(q/2)^{N} \sum_{(A, B) \text{ principal}}\Xi_{A, B}.
\]
The rest of the proof follows from \eqref{def:Lsum} and \Cref{lem:erroresti}.
\end{proof}

Theorem~\ref{MT1'} provides a precise asymptomatic formula for $c_q(\cB_d)$ for even $q$. For example, if $q=4$, then $\thre(q)=3$, so
\[
c_4(\cB_d) = (1 + o(1))(q/2)^{N} \sum_{(A, B) \text{ principal}}\exp\left(L_{A, B}(1) + L_{A, B}(2)\right);
\]
if $q=6$, then $\thre(q)=4$, and therefore
\[
c_6(\cB_d) = (1 + o(1))(q/2)^{N} \sum_{(A, B) \text{ principal}}\exp\left(L_{A, B}(1) + L_{A, B}(2) + L_{A, B}(3)\right).
\]
Note that by the definition of cluster and $L_{A, B}(k)$, given $q$ and a fixed integer $k$, one can obtain an explicit formula for $L_{A, B}(k)$ in a finite time, and so for $c_q(\cB_d)$.

By symmetry, we simply denote $L_{A, B}(k)$ by $L_k$ for all principal partitions $(A, B)$. We next provide an explicit computation of $L_1$ and $L_2$. Then \Cref{realMT} will follow from \Cref{MT1'}, together with the combination of \eqref{L1.compute}, \eqref{L2.compute}, and \eqref{eps.bound}. We remark that the computation of $L_k$ for any fixed $k$ will follow from a similar strategy.

\begin{example}[Computing $L_1$]
Every polymer of size 1 is a single vertex of $\mathcal{L}_d$ or $\mathcal{L}_{d-1}$. There are $N$ of them, and each has weight $\displaystyle \frac{(q/2)(q/2-1)^d}{(q/2)^{d+1}}=(1 - 2/q)^d$. There is only one type of cluster of size 1, which consists of a polymer of size 1, with Ursell function 1. Therefore we have 
\beq{L1.compute} L_1=N(1 - 2/q)^d.\enq
\end{example}

\begin{example}[Computing $L_2$]
A polymer of size 2 is either a set of two vertices of $\mathcal{L}_d$ (or $\mathcal{L}_{d-1}$) sharing a common neighbor, or an edge of $\cB_d$.
For the first type, there are $N\binom{d}{2}$ of them and each has weight 
\[\frac{(q/2)(q/2-1)^{2d-1}}{(q/2)^2(q/2)^{2d-1}} + \frac{(q/2)(q/2-1)(q/2-1)^{2d-2}(q/2-2)}{(q/2)^2(q/2)^{2d-1}}=(1 - 2/q)^{2d}. \]
For the second type, there are $Nd/2$ of them and each has weight $\displaystyle \frac{(q/2)^2(q/2-1)^{2(d-1)}}{(q/2)^{2d}}=(1 - 2/q)^{2(d-1)}$.

There are two types of clusters of size 2. The first type is an ordered pair of adjacent polymers of size 1, whose Ursell function is $-1/2$. The number of such clusters is $N+Nd(d-1) + Nd$, and each of them has weight $(1 - 2/q)^{2d}$.
The second type is a single polymer of size 2 with Ursell function 1. By the above discussion on polymers, $N\binom{d}{2}$ of such clusters has weight $(1 - 2/q)^{2d}$, while the rest $Nd/2$ of them has weight $(1 - 2/q)^{2(d-1)}$.
Therefore, we have
\beq{L2.compute}
\begin{split}
L_2&=-\frac12(N+Nd(d-1) + Nd)(1 - 2/q)^{2d} + N\binom{d}{2}(1 - 2/q)^{2d} + \frac{1}{2}Nd(1 - 2/q)^{2(d-1)}\\
&=N(1 - 2/q)^{2d}\left(\frac{d}{2}(1 - 2/q)^{-2} - \frac{d}{2}-\frac12\right).
\end{split}
\enq
\end{example}

\section{Proof of \Cref{realMT2}}\label{sec.realMT2}
To prove \Cref{realMT2}, we build a new polymer model to count \textit{typical} colorings.
Define a set of polymers
\begin{equation*}
\cP_T:=\left\{\gamma\subseteq V: \gamma \text{ is non-empty and 2-linked, } |\gamma|\leq T(q)-1\right\}.
\end{equation*}
As before, $\gamma$, $\gamma'$ $\in \cP_T$ are adjacent iff $\gamma\cup \gamma'$ is a $2$-linked set; for a principal partition $(A, B)$, the weight function $\we_{A,B}(\gamma)$ is defined as in~\eqref{def:wefun}; the corresponding polymer model partition function is defined as
\[
\Xi(\cP_T, \we_{A, B}):=\sum_{\Lambda\in\Omega_{\cP}}\prod_{\gamma\in\Lambda}\we_{A, B}(\gamma).
\]
Observe that the number of $q$-colorings $f$ such that every $2$-linked component of $X_{A,B}(f)$ is of size \textit{less than} $\thre(q)$ is exactly $\Xi(\cP_T, \we_{A, B})$. Therefore, the theorem will follow from Theorem~\ref{MT1'} once we show
\begin{equation}\label{ineq:typ1}
\Xi(\cP_T, \we_{A, B}) = (1 +o(1))\Xi_{A, B}
\end{equation}
for every principal partitions $(A, B)$.

Let $\mathcal{C}_T$ be the collection of all clusters $(\gamma_1, \gamma_2, \ldots)$ where $\gamma_i\in \cP_T$, and
$L^T_{A, B}(k):=\sum_{\Gamma\in \mathcal{C}_T,\  \lVert \Gamma\rVert=k}\we_{A,B}(\Gamma)$. 
Note that the proof of \Cref{lem:KP} also implies that $\ln\Xi(\cP_T, \we_{A, B})$ converges: indeed, as now we only consider small polymers of constant size, the proof is even simpler. 
Moreover, similarly to \Cref{lem:erroresti}, one can show that for any fixed $t\geq 1$,
\[
\sum_{k=t}^{\infty}|L^T_{A, B}(k)|=\bO\left(Nd^{2(t-1)}\delta_{A, B}^{dt}\right),
\]
and in particular, by the definiton of $T(q)$, see~\eqref{def:threpolysize}, we have
\[
\sum_{k=T(q)}^{\infty}|L^T_{A, B}(k)|=o(1).
\]
Therefore,
\begin{equation}\label{ineq:typ2}
\Xi(\cP_T, \we_{A, B})=\exp\left(\sum_{k=1}^{\infty}L^T_{A, B}(k)\right)=(1 + o(1))\exp\left(\sum_{k=1}^{T(q)-1}L^T_{A, B}(k)\right).
\end{equation}
By the definition of $\lVert \Gamma\rVert$, we have
\begin{equation}\label{ineq:typ3}
L^T_{A, B}(k) = L_{A, B}(k) 
\end{equation}
for every $1\leq k \leq T(q)-1$. Now \eqref{ineq:typ1} follows from \eqref{ineq:typ2}, \eqref{ineq:typ3} and Theorem~\ref{MT1'}.

\medskip

\nin \textbf{Acknowledgement.} This material is based upon work supported by the National Science Foundation
under Grant No. DMS-1928930 and the National Security Agency under Grant No.
H98230-22-1-0018 while the authors participated in a program hosted by the
Mathematical Sciences Research Institute in Berkeley, California, during the
summer of 2022.

\appendix

\section{Proof of \Cref{lem:defect.shorten}}\label{appendix}

%%%%%%%%%%%%%%%%%%%%%
We define $|\Lambda|:=\sum_{\gamma\in\Lambda}|\gamma|$ for  $\Lambda\in\Omega_{\cP}$.

\begin{lem}\label{lem:defectsize}
Let $(A, B)$ be a principal partition and $\mathbf{\Lambda}_{A, B}$ be drawn from $\nu_{A, B}$ defined in~\eqref{def:rd}. Then there exist a constant $L=L(q)$ such that if $t\geq Ld\delta_{A, B}^d$, then 
\[
\pr(|\mathbf{\Lambda}_{A, B}|\geq tN)\leq \exp\left(-\Omega\left(tN/d\right)\right).
\]
\end{lem}

\begin{proof}
Recall from \eqref{def:genepmodel} the definition of $\tilde{\Xi}_{A, B}$.
By Lemma~\ref{lem:KP}, $\ln\tilde{\Xi}_{A, B}$ converges absolutely; moreover, a similar proof as for Lemma~\ref{lem:erroresti} shows that 
\begin{equation}\label{ineq:cpf}
\ln\tilde{\Xi}_{A, B} = \sum_{\Gamma\in \cC}\tilde{\we}_{A, B}(\Gamma) = \bO\left(N\delta_{A, B}^d\right).
\end{equation}
Observe that
\[
\frac{\tilde{\Xi}_{A, B}}{\Xi_{A, B}} = \frac{1}{\Xi_{A, B}}\sum_{\Lambda\in\Omega_{\cP}}\exp\left(\frac{|\Lambda|}{d}\right)\prod_{\gamma\in\Lambda}\we_{A, B}(\gamma)=\bE\exp\left(\frac{|\mathbf{\Lambda}_{A, B}|}{d}\right),
\]
and therefore this, with~\eqref{ineq:cpf}, implies that there exists an constant $C$ such that 
\[
\bE\exp\left(|\mathbf{\Lambda}_{A, B}|/d\right) \leq \exp\left(\ln\tilde{\Xi}_{A, B}\right)\leq  \exp\left(C\cdot N\delta_{A, B}^d\right).
\]
Take $L$ such that $L > C$. Then by Markov's inequality and $t\geq Ld\delta_{A, B}^d$, we have
\[
\begin{split}
\pr\left(|\mathbf{\Lambda}_{A, B}|\geq tN\right)
&=\pr\left(\exp\left(\frac{|\mathbf{\Lambda}_{A, B}|}{d}\right)\geq \exp\left( \frac{tN}{d}\right)\right)
\leq \bE\exp\left(\frac{|\mathbf{\Lambda}_{A, B}|}{d}\right)\exp\left(-\frac{tN}{d}\right)\\
&\leq  \exp\left(C\cdot N\delta_{A, B}^d\right)\exp\left(-tN/d\right)\leq \exp\left(-\Omega\left(tN/d\right)\right).
\end{split}
\]
\end{proof}

%%%%%%%%%%%%%%%%%%%%%

Define the \textit{flaw} of a coloring $f \in C_q(\cB_d)$ to be the set of vertices at which $f$ does not agree with its closest ground state coloring (breaking ties arbitrarily). 
Indeed, we prove the following strengthened version of Lemma~\ref{lem:defect.shorten} (which is also a crucial lemma for proving Theorem~\ref{notMT}).

\begin{lem}\label{lem:defect}
Let $\mathbf f$ be a random coloring drawn from the distribution $\hat{\mu}_q$ (see Definition~\ref{def:muh}). Denote by $\mathbf D$ the principal partition selected at Step 1, and let $\mathbf \Lambda$ be the random polymer configuration selected at Step 2. Let $\mathbf X$ denote the size of the flaw of $\mathbf f$ and $L$ as in \Cref{lem:defectsize}.
Then the following holds.
\begin{itemize}
\item[(i)] $\pr(\mathbf X = |\mathbf\Lambda|) \geq 1 - \exp\left(-\Omega\left(N/d^2\right)\right)$.

\item[(ii)] If $(A, B)$ is a principal partition and $Ld\delta^d_{A, B} \leq t \leq 1/(4d)$, then
\[
\pr(\mathbf X \geq tN \mid \mathbf D=(A, B)) \leq \exp\left(-\Omega\left(tN/d\right)\right).
\]

\item[(iii)] If $(A, B)$ is a principal partition and $10Ld^2\delta_{A, B}^d \leq s\leq 1$, then
\[
\pr\left(\text{$\mathbf f$ is $s$-balanced w.r.t. $(A, B)$} \mid \mathbf D=(A, B) \right) \geq 1 - \exp\left(-\Omega\left(s^2N\right)\right) - \exp\left(-\Omega\left(sN/d^2\right)\right).
\]

\end{itemize}
\end{lem}

\begin{proof}
Let $(A, B)$ and $t$ be as in (ii). By Lemma~\ref{lem:defectsize}, we have
\begin{equation}\label{ineq:defect1}
\pr\left(|\mathbf\Lambda|\leq tN \mid \mathbf D=(A, B)\right) \geq 1 - \exp\left(-\Omega\left(tN/d\right)\right).
\end{equation}

For short, write $\mathbf{\Lambda}^+$ for $\cup_{\gamma\in\mathbf{\Lambda}}\gamma^+$, and let $|\mathbf{\Lambda}^+|:= \sum_{\gamma\in\mathbf{\Lambda}}|\gamma^+|$.
For each $c\in A$ ($c\in B$, resp.), let $Z(c)$ denote the number of vertices in $\cL_d\setminus\mathbf{\Lambda}^+$ ($\cL_{d-1}\setminus\mathbf{\Lambda}^+$, resp.) that receive color $c$ conditioned on the event that $\mathbf D=(A, B)$. Since in Step 4 the vertices of $\cL_d\setminus\mathbf{\Lambda}^+$ ($\cL_{d-1}\setminus\mathbf{\Lambda}^+$, resp.) are colored $k$ independently with probability $1/|A|$ ($1/|B|$, resp.), we have that
\[
Z(c)\sim \bin\left(|\cL_d\setminus\mathbf{\Lambda}^+|, 1/|A|\right) (\bin\left(|\cL_{d-1}\setminus\mathbf{\Lambda}^+|, 1/|B|\right), \mbox{resp.}).
\]
Note that $|\mathbf{\Lambda}^+|\leq (d+1)|\mathbf{\Lambda}|$ and $t\leq 1/(4d)$. 
Then if $|\mathbf \Lambda|\leq tN$, we have both $|\cL_d\setminus\mathbf{\Lambda}^+|,\ |\cL_{d-1}\setminus\mathbf{\Lambda}^+| \geq N/6$.
Thus by Chernoff's bound, for every color $c\in\cC$, 
\begin{equation}\label{ineq:defect2}
\pr\left(Z(c)\leq \bO(N) \mid \mathbf D=(A, B),\ |\mathbf \Lambda|\leq tN\right) \leq \exp\left(-\Omega\left(N\right)\right).
\end{equation}

For each $c\in A$ ($c\in B$, resp.), let $Y(c)$ denote the number of vertices in $\cL_d$ ($\cL_{d-1}$, resp.) that receive color $c$ conditioned on the event that $\mathbf D=(A, B)$.
Observe that for every $c$,
\begin{equation}\label{ineq:defect3}
0\leq Y(c) - Z(c) \leq |\mathbf{\Lambda}^+|.
\end{equation}
Then by \eqref{ineq:defect2}, \eqref{ineq:defect3} and the union bound, we have
\begin{equation}\label{ineq:defect4}
\pr\left(\min_{c\in\cC}Y(c)=\Omega(N) \mid \mathbf D=(A, B),\ |\mathbf \Lambda|\leq tN\right) \geq 1 - \exp\left(-\Omega\left(N\right)\right).
\end{equation}

Now, suppose that the events $\{\mathbf D=(A, B)\}$, $\{|\mathbf \Lambda|\leq tN\}$, and $\{\min_{c\in\cC}Y(c)=\Omega(N)\}$ all hold. For any principal partition $(A', B')\neq (A, B)$, either $A\setminus A'\neq \emptyset$, or $B\setminus B'\neq \emptyset$.
Since $\min_{c\in\cC}Y(c)=\Omega(N)$, $\mathbf f$ must disagree with $(A', B')$ on $\Omega(N)$ vertices. 
On the other hand, $\mathbf f$ disagrees with $(A, B)$ at $\mathbf \Lambda$, which is of size at most $tN=O(N/d)$. Then $(A, B)$ must be the closest ground state of $\mathbf f$, and thus $\mathbf X=|\mathbf \Lambda|$. Moreover,
\[
\pr\left(\mathbf X=|\mathbf \Lambda|\right) \geq \pr\left(\min_{c\in\cC}Y(c)=\Omega(N), \  |\mathbf \Lambda|\leq tN, \  \mathbf D=(A, B)\right).
\]
Note that for any $(A, B)$, by \eqref{ineq:defect1} and \eqref{ineq:defect4}, we have
\begin{equation}\label{ineq:defect5}
\begin{split}
&\pr\left(\min_{c\in\cC}Y(c)=\Omega(N), \  |\mathbf \Lambda|\leq tN \mid \mathbf D=(A, B)\right) \\
=\ &\pr\left(\min_{c\in\cC}Y(c)=\Omega(N) \mid\  |\mathbf \Lambda|\leq tN, \  \mathbf D=(A, B)\right)
\pr\left(|\mathbf \Lambda|\leq tN \mid \mathbf D=(A, B)\right) \\
\geq\ & 1 - \exp\left(-\Omega\left(tN/d\right)\right).
\end{split}
\end{equation}
Therefore, we conclude that
\[
\pr(\mathbf X \leq tN \mid \mathbf D=(A, B)) \geq \pr\left(\min_{c\in\cC}Y(c)=\Omega(N), \  |\mathbf \Lambda|\leq tN \mid \mathbf D=(A, B)\right) \geq\ 1 - \exp\left(-\Omega\left(tN/d\right)\right),
\]
which leads to (ii).

For (i), apply \eqref{ineq:defect5} with $t=1/(4d)$, we have
\[
\pr\left(\mathbf X=|\mathbf \Lambda|\right) \geq \sum_{(A, B)}\pr\left(\min_{c\in\cC}Y(c)=\Omega(N), \  |\mathbf \Lambda|\leq tN \mid \mathbf D=(A, B)\right) \pr\left(\mathbf D=(A, B)\right)
\geq 1 - \exp\left(-\Omega\left(N/d^2\right)\right).
\]

Let $s$ be as in (iii) and assume that $|\mathbf\Lambda^+|\geq s|\cL_d|/4$.
Recall that $|\mathbf\Lambda^+| \leq (d+1)|\mathbf\Lambda|$, and then we must have $|\mathbf\Lambda|\geq tN$ for $t=s/(10d) \leq 1/(4d)$.
Thus by \eqref{ineq:defect1}, 
\begin{equation}\label{ineq:defect6}
\pr\left(|\mathbf\Lambda^+|\leq s\frac{|\cL_d|}{4} \mid \mathbf D=(A, B)\right) \geq 1 - \exp\left(-\Omega\left(sN/d^2\right)\right).
\end{equation}
Let $c\in A$. Similarly as before, as $|\mathbf\Lambda^+|\leq s|\cL_d|/4$, applying Chernoff's bound on $Z(c)$ gives that with probability at least $1 - \exp\left(-\Omega\left(s^2N\right)\right)$,
\[
(1 - s/4)\frac{|\cL_d\setminus\mathbf\Lambda^+|}{|A|} \leq Z(c)\leq (1 + s/4)\frac{|\cL_d\setminus\mathbf\Lambda^+|}{|A|}.
\]
This, together with \eqref{ineq:defect3}, gives that
\[
\left(\frac{1}{|A|}-s\right)|\cL_d|\leq Y(c) \leq \left(\frac{1}{|A|}+s\right)|\cL_d|.
\]
Then we conclude that for every $c\in A$,
\[
\pr\left(\left|\frac{Y(c)}{|\cL_d|} - \frac{1}{|A|}\right|\leq s \mid |\mathbf\Lambda^+|\leq s\frac{|\cL_d|}{4}, \  \mathbf D=(A, B)\right)
\geq 1 - \exp\left(-\Omega\left(s^2N\right)\right).
\]
The analogous statement holds for $c\in B$. Finally, the above, together with \eqref{ineq:defect6} and the union bound, implies (iii).
\end{proof}

\end{document}